\documentclass[12pt]{amsart}
\usepackage{amssymb, eucal, amsfonts, amsmath, xypic,latexsym}

\def\gr{\mathrm{gr}}
\def\k{{\Bbbk}}

\def\g{{\mathfrak g}}

\def\h{{\mathfrak h}}
\def\a{{\mathfrak a}}
\def\sl{{\mathfrak sl}}

\def\N{{\mathcal N}}
\def\I{{\mathcal I}}

\def\CC{{\mathbb C}}

\def\z{{\mathfrak z}}
\def\m{{\mathfrak m}}

\def\n{{\mathfrak n}}
\def\c{{\mathfrak c}}
\def\p{{\mathfrak p}}
\def\P{{\mathfrak P}}
\def\Z{{\mathbb Z}}
\def\N{{\mathbb N}}
\def\O{{\mathcal O}}
\def\End{\mathop{\fam0 End}}

\def\Lie{\mathop{\fam0 Lie}}
\def\Ad{\mathrm{Ad}\,}
\def\ker{\mathrm{Ker}\,}

\def\Mat{\mathop{\fam0 Mat}\nolimits}
\def\sl{\mathop{\mathfrak{sl}}\nolimits}

\def\ad{\mathrm{ad\,}}

\def\la{\langle}
\def\ra{\rangle}

\theoremstyle{plain}
\newtheorem{theorem}{Theorem}[section]
\newtheorem{corollary}{Corollary}[section]
\newtheorem{prop}{Proposition}[section]
\newtheorem{lemma}{Lemma}[section]

\theoremstyle{definition}
\newtheorem{defn}{Definition}[section]

\newtheorem{conj}{Conjecture}[section]

\theoremstyle{remark}

\newtheorem{rem}{Remark}[section]

\def\subtitle#1. {{\medskip\bf#1\par\nobreak\smallskip}}
\def\proclaim#1. {\medbreak\bgroup\noindent\bf#1. \it}

\def\endproclaim{\egroup
\ifdim\lastskip<\medskipamount\removelastskip\medskip\fi}
\newcount
=0
\def\citedef#1 {\advance\citation by1
  \expandafter\edef\csname#1\endcsname{{\the\citation}}
  \checkendcitedef}
\def\checkendcitedef#1{\ifx#1\endcitedef\else\citedef#1\fi}
\def\cite#1{\csname#1\endcsname}
\citedef  AH BS BK Bo BG Br BrK BrK2 Dix Eis GG  GK Gi Gol Ja JoGK
JoK Ko Lo Lo1 Lo2 McR Mg1 Mg2 Mo P95 P02 P07 P07' P10' P10 R Sch Sk
\endcitedef
\newtoks\nextauth
\newif\iffirstauth
\def\checkendauth#1{\ifx\endauth#1
        \iffirstauth\the\nextauth
        \else{} and \the\nextauth\fi,
    \else\iffirstauth\the\nextauth\firstauthfalse
        \else, \the\nextauth\fi
        \expandafter\auth\expandafter#1\fi}
\def\auth#1 #2 {\nextauth={#1 #2}\checkendauth}
\newif\ifinbook
\newif\ifbookref
\def\nextref#1 {\bookreffalse\inbookfalse
    \bibitem[\cite{#1}]{}
    \firstauthtrue
    \ignorespaces}
\def\paper#1{{\it#1,}}
\def\In#1{\inbooktrue In #1,}
\def\book#1{\bookreftrue{\it#1,}}
\def\journal#1{#1\ifinbook,\fi}
\def\bookseries#1{#1,}
\def\Vol#1{\ifbookref Vol. #1,\else\ifinbook Vol. #1,\else{\bf#1}\fi\fi
    \space\ignorespaces}
\def\nombre#1{no. #1}
\def\publisher#1{#1,}
\def\Year#1{\ifbookref #1.\else\ifinbook #1,\else(#1)\fi\fi
    \space\ignorespaces}
\def\Pages#1{\ifinbook pp. #1.\else #1.\fi}
\begin{document}
\title{Enveloping algebras of Slodowy slices and\\ Goldie rank}
\author{Alexander Premet}
\thanks{\nonumber{\it Mathematics Subject Classification} (2000 {\it revision}).
Primary 17B35. Secondary 17B63, 17B81.}
\address{School of Mathematics, University of Manchester, Oxford Road,
M13 9PL, UK} \email{Alexander.Premet@manchester.ac.uk} \maketitle

\medskip
\begin{center}
{\it Dedicated to Professor T.A.~Springer on the occasion of his
85th birthday}
\end{center}

\bigskip

\begin{abstract}
\noindent Let $U(\g,e)$ be the finite $W$-algebra associated with a
nilpotent element $e$ in a complex simple Lie algebra $\g=\Lie(G)$
and let $I$ be a primitive ideal of the enveloping algebra $U(\g)$
whose associated variety equals the Zariski closure of the nilpotent
orbit $(\Ad G)\,e$. Then it is known that $I\,=\,{\rm
Ann}_{U(\g)}\big(Q_e\otimes_{U(\g,\,e)}V\big)$ for some {\it finite
dimensional} irreducible $U(\g,e)$-module $V$, where $Q_e$ stands
for the generalised Gelfand--Graev $\g$-module associated with $e$.
The main goal of this paper is to prove that the Goldie rank of the
primitive quotient $U(\g)/I$ always divides $\dim V$. For
$\g=\mathfrak{sl}_n$, we use a theorem of Joseph on Goldie fields of
primitive quotients of $U(\g)$ to establish the equality  ${\rm
rk}\big(U(\g)/I\big)\,=\,\dim V$. We show that this equality
continues to hold for $\g\not\cong\mathfrak{sl}_n$ provided that the
Goldie field of $U(\g)/I$ is isomorphic to a Weyl skew-field and use
this result to disprove Joseph's version of the Gelfand--Kirillov
conjecture formulated in the mid-1970s.
\end{abstract}

\bigskip

\section{\bf Introduction}
\subsection{}\label{1.1}
Denote by $G$ a simple, simply connected algebraic group over
$\mathbb C$,  let $(e,h,f)$ be a nontrivial $\mathfrak{sl}_2$-triple
in the Lie algebra $\g=\Lie(G)$, and denote by
$(\,\cdot\,,\,\cdot\,)$ the $G$-invariant bilinear form on $\g$ for
which $(e,f)=1$. Let $\chi\in\g^*$ be such that $\chi(x)=(e,x)$ for
all $x\in\g$ and write $U(\g,e)$ for the enveloping algebra of the
Slodowy slice $e+\text{Ker}\,\ad\,f$ to the adjoint orbit $\O:=(\Ad
G)e$; see [\cite{P02}, \cite{GG}]. Recall that
$U(\g,e)\,=\,({\End}_{\g}\,Q_e)^{\text{op}}$, where $Q_e$ is the
generalised Gelfand--Graev $\g$-module associated with the triple
$(e,h,f)$. The module $Q_e$ is induced from a one-dimensional module
${\mathbb C}_\chi$ over of a nilpotent subalgebra $\m$ of $\g$ whose
dimension equals $d(e):=\frac{1}{2}\dim \O$. The Lie subalgebra $\m$
is $(\ad h)$-stable, all eigenvalues of $\ad h$ on $\m$ are
negative, and $\chi$ vanishes on $[\m,\m]$. The action of $\m$ on
${\mathbb C}_\chi={\mathbb C} 1_\chi$ is given by
$x(1_\chi)=\chi(x)1_\chi$ for all $x\in\m$. The algebra $U(\g,e)$ is
also known as the {\it finite $W$-algebra} associated with the pair
$(\g, e)$ and it shares many remarkable features with the universal
enveloping algebra $U(\g)$. It is worth mentioning that $Q_e$ is
free as a right module over  $U(\g,e)$; see [\cite{Sk}] for detail.

\smallskip

From now on we identify $\g$ with $\g^*$ by using the
$G$-equivariant Killing isomorphism $\g\ni x\mapsto
(x,\,\cdot\,)\in\g^*$. Given a primitive ideal $I$ of $U(\g)$ we
write $\mathcal{VA}(I)$ for the associated variety of $I$. By a
classical result of Lie Theory, proved by Borho--Brylinski in
special cases and by Joseph in general, the variety
$\mathcal{VA}(I)$ coincides with the closure of a nilpotent orbit in
$\g$. If $V$ is a finite dimensional irreducible $U(\g,e)$-module,
then it follows from Skryabin's theorem [\cite{Sk}] that the
$\g$-module $Q_e\otimes_{U(\g,\,e)}V$ is irreducible. Hence the
annihilator $I_V:={\rm Ann}_{U(\g)}\big(Q_e\otimes_{U(\g,\,e)}
V\big)$ is a primitive ideal of $U(\g)$. According to [\cite{P07},
Thm.~3.1(ii)], the variety $\mathcal{VA}(I_V)$ coincides with
Zariski closure of the orbit $\O$.

\smallskip

In [\cite{P07}], the author conjectured that the converse is also
true, i.e. for any primitive ideal $I$ of $U(\g)$ with
$\mathcal{VA}(I)=\overline{\O}$ there exists a finite dimensional
irreducible $U(\g,e)$-module $M$ such that $I=I_M$. This conjecture
was proved by the author in [\cite{P07'}] under a mild technical
assumption on the central character of $I$ (removed in [\cite{P10}])
and by Losev [\cite{Lo}] in general. Yet another proof of the
conjecture was later given by Ginzburg in [\cite{Gi}]. Losev's proof
employed his new construction of $U(\g,e)$ via equivariant Fedosov
quantization, whilst Ginzburg's proof was based of the notion of
Harish-Chandra bimodules for quantized Slodowy slices introduced and
studied in [\cite{Gi}]. The author's proof relied almost entirely on
characteristic $p$ methods.
\subsection{}\label{1.2} Write $\mathcal{X}_\O$ for the set of all primitive ideals
$I$ of $U(\g)$ with $\mathcal{VA}(I)=\overline{\O}$ and denote by
${\rm Irr}\,U(\g,e)$ the set of all isoclasses of finite dimensional
irreducible $U(\g,e)$-modules. For $x\in\g$ set $G_x:=\{g\in
G\,|\,\,(\Ad g)\,x=x\}$. It is well known that the group
$C(e):=G_e\cap G_f$ is reductive and its finite quotient
$\Gamma(e):=C(e)/C(e)^\circ$ identifies naturally with the component
group of the centraliser $G_e$. From the realization of $U(\g,e)$
obtained by Gan--Ginzburg [\cite{GG}] it is immediate that the
algebraic group $C(e)$ acts on $U(\g,e)$ by algebra automorphisms.
Thus, we can twist the module structure $U(\g,e)\times M\rightarrow
M$ of any $U(\g,e)$-module $M$ by an element $g\in C(e)$ to obtain a
new $U(\g,e)$-module, $^{g\!}M$, with underlying vector space $M$
and the $U(\g,e)$-action given by $u\cdot m=g^{-1}(u)\cdot m$ for
all $u\in U(\g,e)$ and $m\in M$. It turns out that if the
$U(\g,e)$-module $M$ is irreducible and $g\in C(e)$, then
$I_M=I_{\,^{g }M}$, so that the primitive ideal $I_M$ depends only
on the isomorphism class of $M$; see [\cite{P10}, 4.8], for example.
We thus obtain a natural surjective map $\varphi_e\colon\, {\rm
Irr}\,U(\g,e)\twoheadrightarrow\,\mathcal{X}_\O$ which assigns to an
isoclass $[M]\in{\rm Irr}\,U(\g,e)$ the primitive ideal
$I_M\in\mathcal{X}_\O$, where $M$ is any representative in $[M]$.
The above discussion shows that the map $\varphi_e$ is well defined
and its fibres are stable under the action of $C(e)$.

\smallskip

By [\cite{P07}, Lemma~2.4], there is an algebra embedding
$\Theta\colon\,U(\Lie\,C(e))\hookrightarrow U(\g,e)$ such that the
differential of the rational action of $C(e)$ on $U(\g,e)$ coincides
with $(\ad\circ\,\Theta)_{\vert\Lie(C(e))}$. As a consequence, every
two-sided ideal of $U(\g,e)$ is stable under the action of the
connected group $C(e)^\circ$. Applying this to the primitive ideals
of finite codimension in $U(\g,e)$ it is easy to observe that the
identity component $C(e)^\circ$ of $C(e)$ acts trivially on ${\rm
Irr}\,U(\g,e)$. We thus obtain a natural action of the finite group
$\Gamma(e)$ on the set ${\rm Irr}\,U(\g,e)$.
\subsection{}\label{1.3} Confirming another conjecture of the
author (first circulated around 2007) Losev proved that each fibre
of $\varphi_e$ is a single $\Gamma(e)$-orbit; see [\cite{Lo1},
Thm.~1.2.2]. This result shows that a generalised Gelfand--Graev
model of $I\in\mathcal{X}_\O$ is {\it almost} unique; in particular,
if $I_M=I=I_{M'}$ for two finite dimensional irreducible
$U(\g,e)$-modules $M$ and $M'$, then necessarily $\dim M=\dim M'$.
The main goal of this paper is to relate the latter number with the
Goldie rank of the primitive quotient $U(\g)/I$.

\smallskip

Let us recall the definition of the Goldie rank of a prime
Noetherian ring $\mathcal{A}$. An element of $\mathcal{A}$ is called
\emph{regular} if it is not a zero divisor in $\mathcal{A}$. By
Goldie's theory, the multiplicative set $S$ of all regular elements
of $\mathcal{A}$ satisfies the left and right Ore conditions.
Therefore, it can be used to form a classical ring of fractions
$\mathcal{Q}(\mathcal{A})=S^{-1}\mathcal{A}$; see [\cite{Dix}, 3.6]
for more detail. The ring $\mathcal{Q}(\mathcal{A})$ is prime
Artinian, hence isomorphic to ${\rm Mat}_n(\mathcal{D})$ for some
$n\in\mathbb{N}$ and some skew-field $\mathcal{D}$. We write $n={\rm
rk}(\mathcal{A})$ and call $n$ the \emph{Goldie rank} of
$\mathcal{A}$. The division ring $\mathcal{D}$ is called the {\it
Goldie field} of $\mathcal{A}$. It is well known that ${\rm
rk}(\mathcal{A})=1$ if and only if $\mathcal{A}$ is a domain. As an
important example, $U(\g)$ admits a classical ring of fractions
which is a skew-field. It is sometimes referred to as the {\it Lie
field} of $\g$ and denoted by $K(\g)$. The $n$th Weyl algebra
$${\bf A}_n(\CC)\,:=\,\langle
X_i,\partial/\partial\,X_i\,|\,\,\,1\le i\le n\rangle_{\CC}$$ is a
Noetherian domain, too, and the classical rings of fractions
$\mathcal{Q}(\mathbb{A}_n(\CC))$, where $n\in\N$, are known as {\it
Weyl skew-fields}.

\smallskip

 More generally, it follows from the Feith--Utumi
theorem that the Goldie rank of $\mathcal A$ coincides with the
maximum value of $k\in\N$ for which there is an $x\in\mathcal A$
with $x^k=0$ and $x^{k-1}\ne 0$ (we adopt the standard convention
that $x^0=1$ for any $x\in \mathcal A$). This is an elegant {\it
internal} characterization of Goldie rank, but it is not very useful
in practice. Recall that a two-sided ideal $I$ of $U(\g)$ is called
{\it completely prime} if ${\rm rk}\big(U(\g)/I\big)=1$, that is, if
$U(\g)/I$ is a domain.

\smallskip

 In [\cite{Lo}], Losev proved that for every
finite dimensional irreducible $U(\g,e)$-module $M$ the inequality
${\rm rk}\big(U(\g)/I_M\big)\big)\le \dim M$ holds. Our first
theorem strengthens this result:

\medskip

\noindent {\bf Theorem~A.}\label{A} {\it Let $M$ be a finite
dimensional irreducible $U(\g,e)$-module and let $I_M={\rm
Ann}_{U(\g)}\,\big(Q_e\otimes_{U(\g,\,e)} M\big)$ be the
corresponding primitive ideal in $\mathcal{X}_{\O}$. Then the Goldie
rank of the primitive quotient $U(\g)/I_M$ divides $\dim M$.}

\medskip

\noindent Our proof of Theorem~{\rm A} relies on reduction modulo
$\P$ in the spirit of [\cite{P07'}] and [\cite{P10}, Sect.~4] and
makes use of the techniques introduced in [\cite{P10'}, Sect.~2].

\smallskip

Denote by ${\mathcal X}_\lambda$ the set of all primitive ideals of
$U(\g)$ with central character $\lambda$. This is a finite set
partially ordered by inclusion. By an old result of Dixmier,
${\mathcal X}_\lambda$ contains a unique maximal element which will
be denoted by $I_{\rm max}(\lambda)$. It was pointed out to the
author by Losev (who attributes the observation to Vogan) that if
$\g=\mathfrak{sp}_{2n}$ and $e$ is a nilpotent element of $\g$
corresponding to partition $(2^n)$ of $2n$, then
$\mathcal{VA}(I_{\rm max}(\rho/2))=\overline{\O}$ and the
multiplicity of $\overline{\O}$ in ${\rm gr}\big(U(\g)/I_{\rm
max}(\rho/2)\big)$ equals $2^{n-1}$ (this follows from computations
made by McGovern in [\cite{Mg1}] and also from [\cite{Mg2},
Thm.~5.14(c)]). Here $\rho$ is the half-sum of positive roots and
$\O$ is the adjoint $G$-orbit of $e$. On the other hand, $I_{\rm
max}(\rho/2)=I_M$ for some finite dimensional irreducible
$U(\g,e)$-module $M$ and combining Theorems 1.3.1(2) and 1.2.2 of
[\cite{Lo1}] one derives that the multiplicity of $\overline{\O}$ in
${\rm gr}\big(U(\g)/I_{\rm max}(\rho/2)\big)$ equals
$[\Gamma(e):\Gamma(e,M)](\dim M)^2$, where $\Gamma(e,M)$ stands for
the stabiliser of the isoclass of $M$ in $\Gamma(e)$. Since the
primitive ideal $I_{\rm max}(\rho/2)$ is completely prime by
[\cite{Mg2}, Thm.~6.15] and $\Gamma(e)\cong\Z/2\Z$, we see that the
integer $q_M:=(\dim M)/{\rm rk}\big(U(\g)/I_M\big)$ does not always
divide the order of $\Gamma(e)$. In particular, it can happen
(outside type $\rm A$) that $\dim M>{\rm rk}\big(U(\g)/I_M\big)$. In
the present case, this inequality can be obtained more directly by
showing that the largest commutative quotient $U(\g,e)^{\rm ab}$ of
$U(\g,e)$ is a polynomial algebra in one variable generated by the
image of a Casimir element of $U(\g)$ and by using the fact that the
completely prime primitive ideal $I_{\rm max}(\rho/2)$ is not
induced; see Remark~\ref{rr} for detail.

\smallskip
It is becoming clear now that the dimensions of irreducible
$U(\g,e)$-modules $M$ are in general more computable than the Goldie
ranks of primitive quotients $U(\g)/I_M$. As an important example,
in [\cite{Lo2}] Losev uses earlier work by Mili{\v c}i{\'c}--Soergel
and Backelin to obtain a formula for multiplicities in the category
$O$ for $W$-algebras $U(\g,e)$ associated with nilpotent elements
$e$ of standard Levi type ($e$ is said to be of {\it standard Levi
type} if it is regular in a Levi subalgebra of $\g$). This gives a
theoretical possibility to compute the dimensions of finite
dimensional irreducible $U(\g,e)$-modules in Kazhdan--Lusztig terms.

\smallskip

There are three nilpotent orbits $\O$ in $\g$ with the property that
for $e\in \O$ the equality ${\rm rk}\big(U(\g)/I_M\big)=\dim M$
holds for any finite dimensional irreducible $U(\g,e)$-module $M$.
First, the zero orbit has this property because $U(\g,0)=U(\g)$ and
all primitive ideals in ${\mathcal X}_{\{0\}}$ have finite
codimension in $U(\g)$. Second, if $e$ lies in the regular nilpotent
orbit in $\g$, then classical results of Kostant on Whittaker
modules show that that the algebra $U(\g,e)$ is isomorphic to the
centre of $U(\g)$ and ${\rm rk}\big(U(\g)/I_M\big) = \dim M=1$ for
any irreducible $U(\g,e)$-module $M$; see [\cite{Ko}]. Third, the
minimal nonzero nilpotent orbit of $\g$ enjoys the above property by
[\cite{P07}, Thm.~1.2(v)]. Our second theorem indicates that the
same could be true for many  (but not all!) nilpotent orbits in
finite dimensional simple Lie algebras.

\smallskip

Let $\mathcal{D}_M$ stand for the Goldie field of the primitive
quotient $U(\g)/I_M$. When $\g=\mathfrak{sl}_n$, Joseph proved that
$\mathcal{D}_M$ is isomorphic to a Weyl skew-field, more precisely,
to the Goldie field of the Weyl algebra
$\mathbf{A}_{d(e)}(\mathbb{C})$; see [\cite{JoK}, Thm.~10.3].

\medskip

\noindent {\bf Theorem~B.} {\it If $\mathcal{D}_M$ is isomorphic to
the Goldie field of $\mathbf{A}_{d(e)}(\mathbb{C})$, then ${\rm
rk}\big(U(\g)/I_M\big)=\dim M$ for any finite dimensional
irreducible $U(\g,e)$-module $M$.}

\medskip

Combining Theorem~{\rm B} with the result of Joseph mentioned above
we see that for $\g=\mathfrak{sl}_n$ (and for $\g=\mathfrak{gl}_n$)
the equality ${\rm rk}\big(U(\g)/I_M\big)=\dim M$ holds for all
nilpotent elements $e\in\g$ and all finite dimensional irreducible
$U(\g,e)$-modules $M$.\footnote{Very recently, Jonathan Brundan has
reproved the equality ${\rm rk}\big(U(\g)/I_M\big)=\dim M$ for
$\g=\mathfrak{gl}_n$ by a characteristic zero argument based on
earlier results of Joseph and the theory of finite $W$-algebras of
type $\rm A$; see [\cite{Br}].}

\smallskip

In view of our earlier remarks Theorem~{\rm B} enables us to
describe the completely prime primitive ideals $I$ of
$U(\mathfrak{sl}_n)$ with $\mathcal{VA}(I)=\overline{\O}$ as exactly
those $I=I_M$ for which $M$ is a one-dimensional $U(\g,e)$-module
(one should also keep in mind here that in type $\rm A$ the
component group $\Gamma(e)$ acts trivially on ${\rm Irr}\,U(\g,e)$).
This description differs from the classical one which is due to
M{\oe}glin [\cite{Mo}]. M{\oe}glin's classification of the
completely prime primitive ideals of $U(\mathfrak{sl}_n)$ stems from
her confirmation of a long-standing conjecture of Dixmier according
to which any completely prime primitive ideal of $U(\g)$, for
$\g=\mathfrak{gl}_n$ or $\mathfrak{sl}_n$, coincides with the
annihilator of a $\g$-module induced from a one-dimensional
representation of a parabolic subalgebra of $\g$. We remark that for
any nilpotent element $e\in \g=\mathfrak{sl}_n$  a complete
description of one-dimensional $U(\g,e)$-modules can be deduced from
[\cite{P10}, 3.8] which, in turn, relies on the Brundan--Kleshchev
description of the finite $W$-algebras for $\mathfrak{gl}_n$ as
truncated shifted Yangians; see [\cite{BrK}].\footnote{Very
recently, Brundan has found a new proof of M{\oe}glin's theorem
relying almost entirely on finite $W$-algebra techniques and the
equality ${\rm rk}\big(U(\g)/I_M\big)=\dim M$ for
$\g=\mathfrak{gl}_n$; see [\cite{Br}].}

\smallskip

More generally, using Theorem~{\rm B} and arguing as in [\cite{P10},
4.9] it is straightforward to see that for $\g=\mathfrak{sl}_n$ and
any $d\in\N$ the set
$\mathcal{X}_\O(d)\,:=\,\{I\in\mathcal{X}_\O\,|\,{\rm
rk}\big(U(\g)/I\big)=d\}$ has a natural structure of a quasi-affine
algebraic variety. There is some hope that in the future one would
be able to combine Theorem~{\rm B} with the main results of
[\cite{BrK2}] to determine the scale factors of all Goldie rank
polynomials for $\g=\mathfrak{sl}_n$.\footnote{This goal has now
been achieved in [\cite{Br}] where Brundan used the equality ${\rm
rk}\big(U(\g)/I_M\big)=\dim M$ for $\g=\mathfrak{gl}_n$  to derive
exact formulae for all Goldie rank rank polynomials in type $\rm
A$.}

\smallskip

At this point it should be mentioned that a conjecture  of Joseph
(put forward in 1976) asserts that the Goldie field of a primitive
quotient of $U(\g)$ is {\it always} isomorphic to a Weyl skew-field;
see [\cite{JoGK}, 1.2] and references therein. It is needless to say
that Joseph's conjecture was inspired by the well known
Gelfand--Kirillov conjecture (from 1966) on the structure of the Lie
field $K(\g)$. Curiously, the latter conjecture fails for $\g$
simple outside types ${\rm A}_n$, ${\rm C}_n$ and ${\rm G}_2$ (see
[\cite{P10'}, Thm.~1]) and remains open in types ${\rm C}_n$ and
${\rm G}_2$ (in type $\rm A$ the conjecture was proved by Gelfand
and Kirillov themselves who made use of very special properties of
the so-called {\it mirabolic} subalgebras of $\mathfrak{sl}_n$; see
[\cite{GK}]). Joseph's conjecture is known to hold for many
primitive quotients outside type $\rm A$; see [\cite{JoGK},
\cite{JoK}] and [\cite{Ja}, Satz~15.24]. When it holds for all
primitive quotients $U(\g)/I$ with $\mathcal{VA}(I)=\overline{\O}$
(as happens for the three orbits mentioned above), one can
parametrise the completely prime primitive ideals in
$\mathcal{X}_\O$ by the points of the affine variety $\big({\rm
Specm}\,U(\g,e)^{\rm ab}\big)/{\Gamma(e)}$.

\smallskip

However, our discussion after the formulation of Theorem~{\rm A}
shows that for $\g=\mathfrak{sp}_{2n}$ with $n\ge 3$ and a nilpotent
element $e\in\g$ with $n$ Jordan  blocks of size $2$ we have that
$\dim M>{\rm rk}\big(U(\g)/I_M\big)$ for some irreducible finite
dimensional $U(\g,e)$-module $M$. Applying Theorem~{\rm B} we now
deduce that for that $M$ the Goldie field $\mathcal{D}_M$ is {\it
not} isomorphic to a Weyl skew-field. This shows that Joseph's
version of the Gelfand--Kirillov conjecture fails for
$\g=\mathfrak{sp}_{2n}$ with $n\ge 3$ (it is known that the
conjecture holds for all simple Lie algebras of rank $2$, but a
detailed proof of this fact is missing in the literature).

\medskip

\noindent{\bf Acknowledgements.} The results presented here were
announced in my talks at the conferences on symmetric spaces (Levico
Terme, 2010), noncommutative algebras (Canterbury, 2010) and
algebraic groups (Nagoya, 2010). I would like to thank Jon Brundan,
Ivan Losev and Toby Stafford for their interest and many useful
discussions. I would like to express my gratitude to the referee for
very careful reading and thoughtful suggestions.

\section{\bf Reducing modulo $\P$ admissible forms of primitive quotients}
\subsection{}\label{2.1} Let $G$ be a simple, simply connected algebraic group over $\mathbb
C$, and $\g=\Lie(G)$. Let $\h$ be a Cartan subalgebra of $\g$ and
$\Phi$ the root system of $\g$ relative to $\h$. Choose a basis of
simple roots $\Pi=\{\alpha_1,\ldots,\alpha_\ell\}$ in $\Phi$, let
$\Phi^+$ be the corresponding positive system in $\Phi$, and put
$\Phi^-:=-\Phi^+$. Let $\g=\n^-\oplus\h\oplus\n^+$ be the
corresponding triangular decomposition of $\g$ and choose a
Chevalley basis ${\mathcal
B}=\{e_\gamma\,|\,\,\gamma\in\Phi\}\cup\{h_\alpha\,|\,\,\alpha\in\Pi\}$
in $\g$. Set ${\mathcal
B}^{\pm}:=\{e_\alpha\,|\,\,\alpha\in\Phi^\pm\}$. Let $\g_\Z$ and
$U_\Z$ denote the Chevalley $\Z$-form of $\g$ and the Kostant
$\Z$-form of $U(\g)$ associated with $\mathcal B$. Given a
$\Z$-module $V$ and a $\Z$-algebra $A$, we write $V_A:=
V\otimes_{\Z}A$.

\smallskip

It follows from the Dynkin--Kostant theory that any nilpotent
$G$-orbit in $\g$ intersects with $\g_\Z$. Take a nonzero nilpotent
element $e\in\g_\Z$ and choose $f,h\in\g_{\mathbb Q}$ such that
$(e,h,f)$ is an $\sl_2$-triple in $\g_{\mathbb Q}$. Denote by
$(\,\cdot\,,\,\cdot\,)$ a scalar multiple of the Killing form
$\kappa$ of $\g$ for which $(e,f)=1$ and define $\chi\in\g^*$ by
setting $\chi(x)=(e,x)$ for all $x\in\g$. Given $x\in\g$ we set
${\mathcal O}(x):=(\Ad G)\cdot x$ and $d(x):=\frac{1}{2}\dim
{\mathcal O}(x)$.

\smallskip

Following [\cite{P07'}, \cite{P10}] we call a finitely generated
$\Z$-subalgebra $A$ of $\mathbb C$ {\it admissible} if
$\kappa(e,f)\in A^\times$ and all bad primes of the root system of
$G$ and the determinant of the Gram matrix of
$(\,\cdot\,,\,\cdot\,)$ relative to a Chevalley basis of $\g$ are
invertible in $A$. It follows from the definition that every
admissible ring $A$ is a Noetherian domain. Moreover, it is well
known (and easy to see) that for every ${\mathfrak P}\in {\rm
Specm}\,A$ the residue field $A/\P$ is isomorphic to ${\mathbb
F}_q$, where $q$ is a $p$-power depending on $\P$.
 We denote by $\Pi(A)$ the set of all primes $p\in
\mathbb N$ that occur this way. It follows from Hilbert's
Nullstellensatz, for example, that the set $\Pi(A)$ contains almost
all primes in $\mathbb N$ (see the proof of Lemma~4.4 in
[\cite{P10}] for more detail).

\smallskip

Let $\g(i)=\{x\in\g\,|\,\,[h,x]=ix\}$. Then
$\g=\bigoplus_{i\in\Z}\,\g(i)$, by the $\mathfrak{sl}_2$-theory, and
all subspaces $\g(i)$ are defined over $\mathbb Q$. Also,
$e\in\g(2)$ and $f\in\g(-2)$. We define a (nondegenerate)
skew-symmetric bilinear form $\la\,\cdot\,,\,\cdot\,\ra$ on $\g(-1)$
by setting $\la x,y\ra:=(e,[x,y])$ for all $x,y\in\g(-1)$. There
exists a basis $B=\{z_1',\ldots,z_s',z_1,\ldots,z_s\}$ of $\g(-1)$
contained in $\g_\mathbb Q$ and such that
$$\la z_i',z_j\ra=\delta_{ij},\qquad\ \la z_i,z_j\ra\,=\,\la z_i',z_j'\ra=0
\qquad\quad\ (1\le i,j\le s).$$ As explained in [\cite{P07'}, 4.1],
after enlarging $A$, possibly, one can assume that
$\g_A=\bigoplus_{i\in\,\Z}\,\g_A(i)$, that each
$\g_A(i):=\g_A\cap\g(i)$ is a freely generated over $A$ by a basis
of the vector space $\g(i)$, and that $B$ is a free basis of the
$A$-module $\g_A(-1)$.

\smallskip

Put $\m:=\g(-1)^0\oplus\sum_{i\le  -2}\,\g(i)$ where $\g(-1)^0$
denotes the $\mathbb C$-span of $z_1',\ldots, z_s'$. Then $\m$ is a
nilpotent Lie subalgebra of dimension $d(e)$ in $\g$ and $\chi$
vanishes on the derived subalgebra of $\m$; see [\cite{P02}] for
more detail. It follows from our assumptions on $A$ that
$\m_A=\g_A\cap\m$ is a free $A$-module and a direct summand of
$\g_A$. More precisely, $\m_A=\g_A(-1)^0\oplus\sum_{i\le
-2}\,\g_A(i)$, where $\g_A(-1)^0= \g_A\cap\g(-1)=
Az_1'\oplus\cdots\oplus Az_s'$. Enlarging $A$ further we may assume
that $e,f\in\g_A$ and that $[e,\g_A(i)]$ and $[f,\g_A(i)]$ are
direct summands of $\g_A(i+2)$ and $\g_A(i-2)$, respectively. Then
$\g_A(i+2)=[e,\g_A(i)]$ for all $i\ge 0$.

\smallskip

Write $\g_e=\Lie(G_e)$ for the centraliser of $e$ in $\g$. As in
[\cite{P02} 4.2, 4.3] we choose a basis $x_1,\ldots,
x_r,x_{r+1},\ldots, x_m$ of the free $A$-module
$\p_A:=\bigoplus_{i\ge 0}\,\g_A(i)$ such that
\begin{itemize}
\item[(a)]
$x_i\in\g_A(n_i)$ for some $n_i\in\Z_+$;

\smallskip

\item[(b)]
$x_1,\ldots, x_r$ is a free basis of the $A$-module $\g_A\cap\g_e$;

\smallskip

\item[(c)] $x_{r+1},\ldots, x_m \in[f,\g_A]$.
\end{itemize}
\subsection{}\label{2.2}
Let $Q_e$ be the generalised Gelfand-Graev $\g$-module associated to
$e$. Recall that $Q_e=U(\g)\otimes_{U(\m)}{\mathbb C}_\chi$, where
${\mathbb C}_\chi={\mathbb C}1_\chi$ is a one-dimensional
$\m$-module such that $x\cdot 1_\chi=\chi(x)1_\chi$ for all
$x\in\m$. Given $({\bf a},{\bf b})\in\Z_+^m\times\Z_+^s$ we let
$x^{\bf a}z^{\bf b}$ denote the monomial $x_1^{a_1}\cdots
x_m^{a_m}z_1^{b_1}\cdots z_s^{b_s}$ in $U(\g)$. Set
$Q_{e,\,A}:=U(\g_A)\otimes_{U(\m_A)}A_\chi$, where $A_\chi=A1_\chi$.
Note that $Q_{e,\,A}$ is a $\g_A$-stable $A$-lattice in $Q_e$ with
$\{x^{\bf i}z^{\bf j}\otimes 1_\chi,\,|\,\,({\bf i},{\bf
j})\in\Z_+^m\times\Z_+^s\}$ as a free basis. Given $({\bf a},{\bf
b})\in\Z_+^m\times Z_+^s$ we set
$$|({\bf a},{\bf b})|_e\,:=\,\,\textstyle{\sum}_{i=1}^m\,a_i(n_i+2)+\textstyle{\sum}_{i=1}^s\,b_i.$$
For ${\bf i}=(i_1,\ldots, i_k)\in\Z_+^k$ set $|{\bf
i}|:=\sum_{j=1}^ki_j$. By [\cite{P02}, Thm.~4.6], the algebra
$U(\g,e):=(\End_\g\,Q_e)^{\rm op}$ is generated over $\mathbb C$ by
endomorphisms $\Theta_1,\ldots,\Theta_r$ such that
\begin{eqnarray}\label{lam} \qquad\ \,\Theta_k(1_\chi)\,=\,\Big(x_k+\sum_{0<|({\bf i},{\bf j})|_e\le
n_k+2}\,\lambda_{{\bf i},\,{\bf j}}^k\,x^{\bf i}z^{\bf
j}\Big)\otimes 1_\chi,\qquad\quad \ 1\le k\le r,\end{eqnarray} where
$\lambda_{{\bf i},\,{\bf j}}^k\in\mathbb Q$ and $\lambda_{{\bf
i},\,{\bf j}}^k=0$ if either $|({\bf i},\,{\bf j})|_e=n_k+2$ and
$|{\bf i}|+|{\bf j}|=1$ or ${\bf i}\ne {\bf 0}$, ${\bf j}={\bf 0}$,
and $i_l=0$ for $l>r$. The monomials
$\Theta_1^{i_1}\cdots\Theta_r^{i_r}$ with $(i_1,\ldots,
i_r)\in\Z_+^r$ form a basis of the vector space $U(\g,e)$.

\smallskip

The monomial $\Theta_1^{i_1}\cdots\Theta_r^{i_r}$ is said to have
{\it Kazhdan degree} $\sum_{i=1}^r\,a_i(n_i+2)$. For $k\in\Z_+$ we
let $U(\g,e)_k$ denote the $\mathbb C$-span of all monomials
$\Theta_1^{i_1}\cdots\Theta_r^{i_r}$ of Kazhdan degree $\le k$. The
subspaces $U(\g,e)_k$, $k\ge 0$, form an increasing exhaustive
filtration of the algebra $U(\g,e)$ called the {\it Kazhdan
filtration}; see [\cite{P02}]. The corresponding graded algebra
$\gr\,U(\g,e)$ is a polynomial algebra in
$\gr\,\Theta_1,\ldots,\gr\,\Theta_r$. It follows from [\cite{P02},
Thm.~4.6] that there exist polynomials $F_{ij}\in{\mathbb
Q}[X_1,\ldots, X_r]$, where $1\le i<j\le r$, such that
\begin{eqnarray}\label{relations}
\qquad\
[\Theta_i,\Theta_j]\,=\,F_{ij}(\Theta_1,\ldots,\Theta_r)\qquad\quad\
\ (1\le i<j\le r).
\end{eqnarray}
Moreover, if $[x_i,x_j]\,\,=\,\,\sum_{k=1}^r \alpha_{ij}^k\, x_k$ in
$\g_e$, then
$$F_{ij}(\Theta_1,\ldots,\Theta_r)\,\equiv\,\sum_{k=1}^r\alpha_{ij}^k\Theta_k
+q_{ij}(\Theta_1,\ldots,\Theta_r)\ \ \ \big({\rm
mod}\,U(\g,e)_{n_i+n_j}\big),$$ where the initial form of $q_{ij}\in
{\mathbb Q}[X_1,\ldots, X_r]$ has total degree $\ge 2$ whenever
$q_{ij}\ne 0$ (as usual, by the initial form of a nonzero polynomial
$f\in {\mathbb Q}[X_1,\ldots, X_r]$ we mean the nonzero component of
smallest degree in the homogeneous decomposition of $f$). By
[\cite{P07'}, Lemma~4.1], the algebra $U(\g,e)$ is generated by
$\Theta_1,\ldots,\Theta_r$ subject to the relations
(\ref{relations}). In what follows we assume that our admissible
ring $A$ contains all $\lambda_{{\bf i},{\bf j}}^k$ in (\ref{lam})
and all coefficients of the $F_{ij}$'s in (\ref{relations}) (due to
the above PBW theorem for $U(\g,e)$ we can view the $F_{ij}$'s as
polynomials in $r=\dim \g_e$ variables with coefficients in $\mathbb
Q$).
\subsection{} \label{2.3}
Let $N_\chi$ denote the left ideal of $U(\g)$ generated by all
$x-\chi(x)$ with $x\in \m$. Then $Q_e\cong U(\g)/N_\chi$ as
$\g$-modules. As $N_\chi$ is a $(U(\g), U(\m))$-bimodule, the fixed
point space $(U(\g)/N_\chi)^{\ad \m}$ carries a natural algebra
structure given by $(x+N_\chi)\cdot (y+N_\chi)=xy+N_\chi$ for all
$x,y\in U(\g)$. Moreover, $U(\g)/N_\chi\cong Q_e$ as $\g$-modules
via the $\g$-module map sending $1+N_\chi$ to $1_\chi$, and
$(U(\g)/N_\chi)^{\ad \m}\cong U(\g,e)$ as algebras. Any element of
$U(\g,e)$ is uniquely determined by its effect on the generator
$1_\chi\in Q_e$ and the canonical isomorphism between
$(U(\g)/N_\chi)^{\ad \m}$ and $U(\g,e)$ is given by $u\mapsto
u(1_\chi)$ for all $u\in(U(\g)/N_\chi)^{\ad \m}$. This isomorphism
is defined over $A$. In what follows we will often identify $Q_e$
with $U(\g)/N_\chi$ and $U(\g,e)$ with $(U(\g)/N_\chi)^{\ad \m}$.

\smallskip

Let $U(\g)=\bigcup_{j\in\Z}\, {\sf K}_jU(\g)$ be the Kazhdan
filtration of $U(\g)$; see [\cite{GG}, 4.2]. Recall that ${\sf
K}_jU(\g)$ is the $\mathbb C$-span of all products $x_1\cdots x_t$
with $x_i\in\g(n_i)$ and $\sum_{i=1}^t\, (n_i+2)\le j$. The Kazhdan
filtration on $Q_e$ is defined by ${\sf K}_jQ_e:=\pi({\sf
K}_jU(\g))$ where $\pi\colon U(\g)\twoheadrightarrow U(\g)/{\mathcal
I}_\chi$ is the canonical homomorphism. It turns $Q_e$ into a
filtered $U(\g)$-module. The Kazhdan grading of $\gr\, Q_e$ has no
negative components, and the Kazhdan filtration of $U(\g,e)$ defined
in \ref{2.2} is nothing but the filtration of
$U(\g,e)=(U(\g)/N_\chi)^{\ad \m}$ induced from the Kazhdan
filtration of $Q_e$ through the embedding $(U(\g)/N_\chi)^{\ad
\m}\hookrightarrow Q_e$; see [\cite{GG}] for more detail.

\smallskip

Let $U(\g_A,e)$ denote the $A$-span of all monomials
$\Theta_1^{i_1}\cdots\Theta_r^{i_r}$ with
$(i_1,\ldots,i_r)\in\Z_+^r$. Our assumptions on $A$  guarantee that
$U(\g_A,e)$ is an $A$-subalgebra of $U(\g,e)$ contained in
$(\End_{\g_A}\,Q_{e,\,A})^{\rm op}$. It is immediate from the above
discussion that $Q_{e,A}$ identifies with the $\g_A$-module
$U(\g_A)/N_{\chi,A},$ where $N_{\chi,A}$ stands for the left ideal
of $U(\g_A)$ generated by all $x-\chi(x)$ with $x\in \m_A$. Hence
$U(\g_A,e)$ embeds into the $A$-algebra
$\big(U(\g_A)/N_{\chi,A}\big)^{\ad\,\m_A}$. As $Q_{e,\,A}$ is a free
$A$-module with basis consisting of all $x^{\bf i}z^{\bf j}\otimes
1_\chi$ with $({\bf i},{\bf j})\in\Z_+^m\times\Z_+^s$ we have that
\begin{eqnarray}\label{A-algebras}
U(\g_A,e)\,=\,({\End}_{\g_A}\,Q_{e,\,A})^{\rm op}\,\cong\,
\big(U(\g_A)/N_{\chi,A}\big)^{\ad\,\m_A}.
\end{eqnarray}
Also, $Q_{\chi,\,A}$ is free as a right $U(\g_A,e)$-module; see
[\cite{P10}, 2.3] for detail.
\subsection{}\label{2.4}
We now pick $p\in\Pi(A)$ and denote by $\k$ an algebraic closure of
${\mathbb F}_p$. Since the form $(\,\cdot\,,\,\cdot\,)$ is
$A$-valued on $\g_A$, it induces a symmetric bilinear form on the
Lie algebra $\g_\k\cong\g_A\otimes_A\k$. We use the same symbol to
denote this bilinear form on $\g_\k$. Let $G_\k$ be the simple,
simply connected algebraic $\k$-group with hyperalgebra
$U_\k=U_\Z\otimes_Z\k$. Note that $\g_\k=\Lie(G_\k)$ and the form
$(\,\cdot\,,\,\cdot\,)$ is $(\Ad\, G_\k)$-invariant and
nondegenerate. For $x\in\g_A$ we set $\bar{x}:=x\otimes 1$, an
element of $\g_\k$. To ease notation we identify $e, f$ with the
nilpotent elements $\bar{e},\bar{f}\in\g_\k$ and $\chi$ with the
linear function $(e,\,\cdot\,)$ on $\g_\k$.

The Lie algebra $\g_\k=\Lie(G_\k)$ carries a natural $[p]$-mapping
$x\mapsto x^{[p]}$ equivariant under the adjoint action of $G_\k$.
 The subalgebra of $U(\g_\k)$
generated by all $x^p-x^{[p]}\in U(\g_\k)$ is called the $p$-{\it
centre} of $U(\g_\k)$ and denoted $Z_p(\g_\k)$ or $Z_p$ for short.
It is immediate from the PBW theorem that $Z_p$ is isomorphic to a
polynomial algebra in $\dim\g$ variables and $U(\g_\k)$ is a free
$Z_p$-module of rank $p^{\dim \g}$. For every maximal ideal $J$ of
$Z_p$ there is a unique linear function $\eta=\eta_J\in\g_\k^*$ such
that
$$J=\,\langle x^p-x^{[p]}-\eta(x)^p1\,|\,\,\,x\in\g_\k\rangle.$$
Since the Frobenius map of $\k$ is bijective, this enables us to
identify the maximal spectrum ${\rm Specm}\,Z_p$ with $\g_\k^*$.

Given $\xi\in\g_\k^*$ we denote by $I_\xi$ the two-sided ideal of
$U(\g_\k)$ generated by all $x^p-x^{[p]}-\xi(x)^p1$ with
$x\in\g_\k$, and set $U_\xi(\g_\k):=U(\g_\k)/I_\xi$. The algebra
$U_\xi(\g_\k)$ is called the {\it reduced enveloping algebra} of
$\g_\k$ associated to $\xi$. The preceding remarks imply that
$\dim_\k U_\xi(\g_\k)=p^{\dim\g}$ and $I_\xi \cap Z_p=J_\xi$, the
maximal ideal of $Z_p$ associated with $\xi$. Every irreducible
$\g_\k$-module is a module over $U_\xi(\g_\k)$ for a unique
$\xi=\xi_V\in\g_\k^*$. The linear function $\xi_V$ is called the
{\it $p$-character} of $V$; see [\cite{P95}] for more detail. By
[\cite{P95}], any irreducible $U_\xi(\g_\k)$-module has dimension
divisible by $p^{(\dim \g-\dim \z_\xi)/2},$ where
$\z_\xi=\{x\in\g_\k\,|\,\,\xi([x,\g_\k])=0\}$ is the stabiliser of
$\xi$ in $\g_\k$. We denote by $Z_{G_\k}(\xi)$ the coadjoint
stabiliser of $\xi$ in $G_\k$.

\subsection{}\label{2.5}
For $i\in\Z$, set $\g_\k(i):=\g_A(i)\otimes_A\k$ and put
$\m_{\k}:=\m_A\otimes_A\k$. Our assumptions on $A$ yield that the
elements $\bar{x}_1,\ldots, \bar{x}_r$ form a basis of the
centraliser $(\g_\k)_e$ of $e$ in $\g_\k$ and that $\m_{\k}$ is a
nilpotent subalgebra of dimension $d(e)$ in $\g_\k$. Set
$Q_{e,\,\k}:=U(\g_\k)\otimes_{U(\m_\k)}\k_{\chi},$ where
$\k_\chi=A_\chi\otimes_A\k\,=\,\k1_{\chi}$. Clearly, $\k1_{\chi}$ is
a one-dimensional $\m_\k$-module with the property that
$x(1_\chi)=\chi(x)1_\chi$ for all $x\in\m_\k$. It follows from our
discussion in \ref{2.2} and \ref{2.3} that $Q_{e,\,\k}\cong
Q_{e,\,A}\otimes_A\k$ as modules over $\g_\k$ and $Q_{e,\,\k}$ is a
free right module over the $\k$-algebra
$$U(\g_\k,e):=U(\g_A,e)\otimes_A\k.$$ Thus we may identify $U(\g_\k,e)$
with a subalgebra of
$\widehat{U}(\g_\k,e):=\big(\End_{\g_\k}\,Q_{e,\,\k}\big)^{\rm op}$.
The algebra $U(\g_\k,e)$ has $\k$-basis consisting of all monomials
$\bar{\Theta}_1^{i_1}\cdots \bar{\Theta}_{r}^{i_r}$ with
$(i_1,\ldots, i_r)\in\Z_+^r$, where $\bar{\Theta}_i:=\Theta_i\otimes
1\in U(\g_A,e)\otimes_A\k$. Given $g\in A[X_1,\ldots, X_n]$ we write
$^p g$ for the image of $g$ in the polynomial algebra
$\k[X_1,\ldots, X_n]\,=\,A[X_1,\ldots, X_n]\otimes_A\k$. Since all
polynomials $F_{ij}$ are in $A[X_1,\ldots, X_r]$, it follows from
the relations (\ref{relations}) that
\begin{eqnarray}\label{relations'}
\qquad\
[\bar{\Theta}_i,\bar{\Theta}_j]\,=\,{^p\!}F_{ij}(\bar{\Theta}_1,\ldots,\bar{\Theta}_r)\qquad\quad\
\ (1\le i<j\le r).
\end{eqnarray}
By [\cite{P10}, Lemma~2.1],  the algebra $U(\g_\k,e)$ is generated
by the elements $\bar{\Theta}_1,\ldots,\bar{\Theta}_r$ subject to
the relations (\ref{relations'}).

\smallskip

Let $\g_A^*$ be the $A$-module dual to $\g_A$ and let $\m_A^\perp$
denote the set of all linear functions on $\g_A$ vanishing on
$\m_A$. By our assumptions on $A$, this is a free $A$-submodule and
a direct summand of $\g_A^*$. Note that $\m_A^\perp\otimes_A\mathbb
C$ and $\m_A^\perp\otimes_A\k$ identify naturally with with the
annihilators $\m^\perp:=\{f\in\g^*\,|\,\,f(\m)=0\}$ and
$\m_\k^\perp:=\{f\in\g^*_\k\,|\,\,f(\m_\k)=0\}$, respectively.

\smallskip

Following [\cite{P10}], for $\eta\in\chi+\m_\k^\perp$ we set
$Q_{e}^{\eta}:=Q_{e,\,\k}/I_\eta Q_{e,\,\k}$. By construction,
$Q_{e}^{\eta}$ is a $\g_\k$-module with $p$-character $\eta$.  Each
$\g_\k$-endomorphism $\bar{\Theta_i}$ of $Q_{e,\,\k}$ preserves
$I_\eta Q_{e,\,\k}$, hence induces a $\g_\k$-endomorphism of
$Q_{e}^{\eta}$ which we denote by $\theta_i$. We write
$U_\eta(\g_\k,e)$ for the algebra
$\big(\End_{\g_\k}\,Q_{e}^{\eta}\big)^{\rm op}$. Since the
restriction of $\eta$ to $\m_\k$ coincides with that of $\chi$, the
left ideal of $U(\g_\k)$ generated by all $x-\eta(x)$ with
$x\in\m_\k$ equals $N_{\chi,\,\k}:=N_{\chi,\,A}\otimes_A\k$ and
$\k_\chi=\k_\eta$ as $\m_\k$-modules. We denote by $N_{\eta,\,\chi}$
the left ideal of $U_\eta(\g_\k)$ generated by all $x-\chi(x)$ with
$x\in\m_\k$. The following are proved in [\cite{P10}, 2.6]:

\smallskip

\begin{itemize}
\item[(a)] $Q_{e}^{\eta}\,\cong\,
U_{\eta}(\g_\k)\otimes_{U_{\eta}(\m_\k)}\k_\chi$ as $\g_\k$-modules;

\smallskip

\item[(b)] $U_{\eta}(\g_\k,e)\,\cong\,\big(U_{\eta}(\g_\k)/
U_{\eta}(\g_\k)N_{\eta,\,\chi}\big)^{\ad\m_\k}$;

\smallskip
\item[(c)] $Q_{e}^{\eta}$ is a projective generator for
$U_{\eta}(\g_\k)$ and
$U_{\eta}(\g_\k)\,\cong\,\Mat_{p^{d(e)}}\big(U_{\eta}(\g_\k,e)\big)$;

\smallskip

\item[(d)] the monomials $\theta_1^{i_1}\cdots\theta_r^{i_r}$ with
$0\le i_k\le p-1$ form a $\k$-basis of $U_{\eta}(\g_\k,e)$.
\end{itemize}

\smallskip

\noindent Moreover, a Morita equivalence between
$U_\eta(\g_\k,e)$-{\sf mod} and $U_\eta(\g_\k)$-{\sf mod} in
part~(b) is given explicitly by the functor that sends a finite
dimensional $U_\eta(\g_\k,e)$-module $W$ to the
$U_\eta(\g_\k)$-module
$\widetilde{W}=Q_e^\eta\otimes_{U_\eta(\g_\k,\,e)}W$, whilst the
quasi-inverse functor from $U_\eta(\g_\k)$-{\sf mod} to
$U_\eta(\g_\k,e)$-{\sf mod} sends a $U_\eta(\g_\k)$-module
$\widetilde{W}$ to its subspace
$$W={\rm Wh}_\eta(\widetilde{W}):=\{v\in \widetilde{W}\,|\,\,x.v=\eta(x)v\ \ \mbox{for all
}\,x\in\m_\k\}.$$

Recall from \ref{2.1} the $A$-basis $\{x_1,\ldots, x_r,
x_{r+1},\ldots,x_m\}$ of $\p_A$ and set
$$X_i=\left\{
\begin{array}{ll}
z_i&\mbox{if $\ 1\le i\le s$},\\
x_{r-s+i}&\mbox{if $\ s+1\le i\le m-r+s$}.
\end{array}\right .$$
For ${\bf a}\in \Z_+^{d(e)}$, put $X^{\bf a}:=X_1^{a_1}\cdots
X_{d(e)}^{a_{d(e)}}$ and $\bar{X}^{\bf a}:=\bar{X}_1^{a_1}\cdots
\bar{X}_{d(e)}^{a_{d(e)}}$, elements of $U(\g_A)$ and $U(\g_\k)$,
respectively. By [\cite{P07'}, Lemma~4.2(i)], the vectors $X^{\bf
a}\otimes 1_\chi$ with ${\bf a}\in\Z^{d(e)}$ form a free basis of
the right $U(\g_A,e)$-module $Q_{e,\,A}$. Let $\a_\k$ be the
$\k$-span of $\bar{X}_1,\ldots, \bar{X}_{d(e)}$ in $\g_\k$ and put
$\widetilde{\a}_\k:=\a_\k\oplus\z_\chi$. Note that $
\a_\k\,=\,\{x\in\widetilde{\a}_\k\,|\,\,\,(x,\ker\ad f)=0\}$. Since
$\chi$ vanishes on $\widetilde{\a}_\k$, we may identify the
symmetric algebra $S(\widetilde{\a}_\k)$ with the coordinate ring
$\k[\chi+\m_\k^\perp]$ by setting $x(\eta):=\eta(x)$ for all
$x\in\widetilde{\a}_\k$ and $\eta\in\chi+\m_\k^\perp$ and extending
to $S(\widetilde{\a}_\k)$ algebraically.

\smallskip

Given a subspace $V\subseteq \g_\k$ we denote by $Z_p(V)$ the
subalgebra of the $p$-centre $Z(\g_\k)$ generated by all
$x^p-x^{[p]}$ with $x\in V$. Clearly, $Z_p(V)$ is isomorphic to a
polynomial algebra in $\dim_\k V$ variables. Let $\rho_\k$ denote
the representation of $U(\g_\k)$ in $\End_\k Q_{e,\,\k}$.

\smallskip

In [\cite{P10}, 2.7] we proved the following:
\begin{theorem}\label{U-hat}
The algebra $\widehat{U}(\g_\k,e)$ is generated by $U(\g_\k,e)$ and
$\rho_\k(Z_p)\cong Z_p(\widetilde{\a}_\k)$. Moreover,
$\widehat{U}(\g_\k,e)$ is a free $\rho_\k(Z_p)$-module with basis
$\{\bar{\Theta}_1^{a_1}\cdots\bar{\Theta}_r^{a_r}\,|\,\,0\le a_i\le
p-1\}$ and $\widehat{U}(\g_\k,e)\cong U(\g_\k,e)\otimes Z_p(\a_\k)$
as $\k$-algebras.
\end{theorem}
Combining [\cite{P10}, Thm.~2.1(ii)] with [\cite{P10},
Lemma~2.2(iv)] it is straightforward to see that $Q_{e,\,\k}$ is a
free right $\widehat{U}(\g_\k,e)$-module with basis
$\{\bar{X}_1^{a_1}\cdots \bar{X}_{d(e)}^{a_{d(e)}}\otimes
1_\chi\,|\,\,0\le a_i\le p-1\}$ and $U_\eta(\g_\k,e)\,\cong\,
\widehat{U}(\g_\k,e)\otimes_{Z_p(\widetilde{\a}_\k)}\k_\eta$ for
every $\eta\in\chi+\m_\k^\perp$. (The algebra
$Z_p(\widetilde{\a}_\k)$ acts on $\k_\eta=\k1_\eta$ by the rule
$(x^p-x^{[p]})(1_\eta)=\eta(x)^p$ for all $x\in\widetilde{\a}_\k$.)
\subsection{}\label{2.6} From now on we fix a primitive ideal $\I$ of $U(\g)$
with $\mathcal{VA}(\I)=\overline{\O}$. The affine variety
$\mathcal{VA}(\I)$ is the zero locus in $\g^*\cong\g$ of the $(\Ad
G)$-invariant ideal $\gr\, \I$ of $S(\g)=\gr\, U(\g)$. As we
identify $\g$ with $\g^*$ by using the Killing isomorphism $\kappa$,
our assumption on $\I$ simply means that the open $({\rm
Ad}^*\,G)$-orbit of $\mathcal{VA}(\I)$ contains $\chi$. We know from
[\cite{Lo}, Thm.~1.2.2], [\cite{P10}, Thm.~4.2] and [\cite{Gi},
Thm.~4.5.2] that $\I={\rm
Ann}_{U(\g)}\big(Q_e\otimes_{U(\g,\,e)}M\big)$ for some finite
dimensional $U(\g,e)$-module $M$. We choose a $\CC$-basis basis
$E=\{m_1,\ldots,m_l\}$ of $M$ and denote by $\tilde{A}$ the
$A$-subalgebra of $\CC$ generated by the coefficients of the
coordinate vectors of all $\Theta_i(m_j)\in M$ with respect to $E$.
By construction, the ring $\tilde{A}$ is admissible and the
$\tilde{A}$-span of $E$ is a $U(\g_A,e)$-stable $\tilde{A}$-lattice
in $M$. Thus, after replacing $A$ by $\tilde{A}$ we may assume that
the lattice $M_A:=Am_1\oplus\cdots\oplus Am_l$ of $M$ is
$U(\g_A,e)$-stable. We write $\tau_A$ for the corresponding
representation of $U(\g_A,e)$ in $\End M_A$. Our discussion in
\ref{2.3} and \ref{2.5} then shows that the $\g$-module
$\widetilde{M}:=Q_e\otimes_{U(\g,\,e)}\,M$ contains a $\g_A$-stable
$A$-lattice with basis $\{X^{\bf a}\otimes m_i\,|\,\,{\bf a}\in
\Z_+^{d(e)},\ 1\le i\le l\}$; we call it $\widetilde{M}_A$. Note
that $\widetilde{M}_A\cong Q_{e,\,A}\otimes_{U(\g_A,\,e)}M_A$ as
$\g_A$-modules. For $p\in\Pi(A)$, the $\g_\k$-module
$\widetilde{M}_\k$ has $\k$-basis $\{\bar{X}^{\bf
a}\otimes\bar{m}_i\,|\,\,{\bf a}\in \Z_+^{d(e)},\ 1\le i\le l\}$,
where $\bar{m}_i=m_i\otimes 1$. Also, $\widetilde{M}_\k\cong
Q_{e,\,\k}\otimes_{U(\g_\k,\,e)}M_\k$ as $\g_\k$-modules.

\smallskip

For $1\le i,j\le l$ denote by $E_{i,j}$ the endomorphism of $M$ such
that $E_{i,j}(m_k)=\delta_{j,k}m_i$ for all $1\le k\le l$. As $M$ is
an irreducible $U(\g,e)$-module, we may assume, after enlarging $A$
further if necessary, that all $E_{i,j}$'s are in the image of
$U(\g_A,e)$ in $\End M$. Thus we may assume that for every
$p\in\Pi(A)$ the $U(\g_\k,e)$-module $M_\k$ is irreducible. We
mention that $U(\g_\k,e)$ acts on $M_\k$ via the representation
$\tau_\k=\tau_A\otimes 1$. By Theorem~\ref{U-hat},
$\widehat{U}(\g_\k,e)\cong U(\g_\k,\,e)\otimes_\k Z_p(\a_\k)$ as
$\k$-algebras. Therefore, for any linear function $\psi$ on $\a_\k$
there is a unique representation
$\widehat{\tau}_{\k,\,\psi}\colon\,\widehat{U}(\g_\k,e)\to\, \End
M_\k$ with $\widehat{\tau}_{\k,\,\psi}(x^p-x^{[p]})=\psi(x)^p{\rm
Id}$ for all $x\in\a_\k$
 whose restriction to
$U(\g_\k,e)\hookrightarrow\widehat{U}(\g_\k,e)$ coincides with
$\tau_\k$. Since the representation $\widetilde{\tau}_{\k,\,\psi}$
is irreducible and $Z_p(\widetilde{\a}_\k)$ is a central subalgebra
of $\widehat{U}(\g_\k,e)$, the linear function $\psi$ extends
uniquely to a linear function $\Psi$ on $\widetilde{\a}_\k$ such
that $\widehat{\tau}_{\k,\,\psi}(x^p-x^{[p]})=\Psi(x)^p{\rm Id}$ for
all $x\in\widetilde{\a}_\k$. As $\g_\k=\m_\k\oplus
\widetilde{\a}_\k$, we can extend $\Psi$ to a linear function on
$\g_\k$ by setting $\Psi(x)=\chi(x)$ for all $x\in\m_\k$. By
construction, $\Psi\in \chi+\m_\k^\perp$ and
$\Psi\vert_{\a_\k}=\psi$.

\smallskip

We now set
$\widetilde{M}_{\k,\,\Psi}:=\,\widetilde{M}_\k/I_{\Psi}\widetilde{M}_\k$,
a $\g_\k$-module with $p$-character $\Psi$. The definition of $\Psi$
and our discussion in \ref{2.5} show that
\begin{eqnarray*}
\widetilde{M}_{\k,\,\Psi}&\cong&
\widetilde{M}_{\k}\otimes_{Z_p(\g_\k)}\k_\Psi\,=\,
\big(Q_{e,\,\k}\otimes_{U(\g_\k,\,e)}M_\k\big)\otimes_{Z_p(\m_\k)\otimes
Z_p(\widetilde{\a}_\k)}\k_\Psi\\
&\cong&\big(Q_{e,\,\k}\otimes_{U(\g_\k,\,e)}M_\k\big)\otimes_{Z_p(\widetilde{\a}_\k)}\k_\Psi\,\cong\,
Q_{e,\,\k}\otimes_{\widehat{U}(\g_\k,\,e)}M_\k\,
\cong\,Q_e^\Psi\otimes_{U_\Psi(\g_\k,\,e)}M_\k,
\end{eqnarray*}
where we view $M_\k$ as a $\widehat{U}(\g_\k,e)$-module via the
representation $\widehat{\tau}_{\k,\,\psi}$. This implies that under
our assumptions on $A$ and $\Psi$ the $U_\Psi(\g_\k)$-module
$\widetilde{M}_{\k,\,\Psi}$ is irreducible and has dimension
$lp^{d(e)}$; see \ref{2.5} for more detail.
\smallskip
\begin{rem}\label{rrr}
One can prove that the linear functions $\Psi$ constructed in this
subsection form a single orbit under the action of the connected
unipotent subgroup $\mathcal{M}_\k$ of $G_\k$ such that $\Ad
\mathcal{M}_\k$ is generated by all linear operators $\exp \ad x$
with $x\in\m_\k$. Indeed, the group $\mathcal{M}_\k$ preserves the
left ideal $U(\g_\k)N_{\chi,\,\k}$ and hence acts on both
$Z_p(\widetilde{\a}_\k)=\rho_\k(Z_p(\g_\k))$ and
$\widehat{U}(\g_\k,e)=(U(\g_\k)/U(\g_\k)N_{\chi,\,\k})^{\ad\,\m_\k}$.
The rational action of $\mathcal{M}_\k$ on $Q_{e,\,\k}$ is obtained
by reducing modulo $\P$ the natural action on $Q_{e,\,A}$ of the
unipotent subgroup $\mathcal{M}_A$ of $G$ such that $\Ad
\mathcal{M}_A$ is generated by all inner automorphisms $\exp \ad x$
with $x\in\m_A$. From this it follows that $U(\g_\k,e)\subseteq
\widehat{U}(\g_\k,e)^{\mathcal{M}_\k}$ (one should keep in mind here
that $U(\g_\k,e)$ is generated by $\bar{\Theta}_1,\ldots,
\bar{\Theta}_r$ and $p\gg 0$). As we identify $S(\widetilde{\a}_\k)$
with $\k[\chi+\m_\k^\perp]$, we may regard the
$\mathcal{M}_\k$-algebra $Z_p(\widetilde{\a}_\k)$ as the coordinate
algebra of the Frobenius twist $(\chi+\m_\k^\perp)^{(1)}\subset
(\g_\k^*)^{(1)}$ of $\chi+\m_\k^\perp$; see [\cite{P10'}, 3.4] for
more detail. The natural action of $\mathcal{M}_\k$  on
$(\chi+\m_\k^\perp)^{(1)}$ is a Frobenius twist of the coadjoint
action of $\mathcal{M}_\k$ on $\chi+\m_\k^\perp$. By
Theorem~\ref{U-hat}, $\widehat{U}(\g_\k,e)$ is a free
$Z_p(\widetilde{\a}_\k)$-module with basis consisting of elements
from $U(\g_\k,e)$. From this we deduce that
$\widehat{U}(\g_\k,e)^{\mathcal{M}_\k}=U(\g_\k,e)$ and
$Z_p(\widetilde{\a}_\k)\cap U(\g_\k,e)=
Z_p(\widetilde{\a}_\k)^{\mathcal{M}_\k}$. On the other hand,
[\cite{P10}, Lemma~3.2] entails that each fibre of the categorical
quotient $\chi+\m_\k^\perp\to(\chi+\m_\k^\perp)/\!\!/\mathcal{M}_\k$
induced by inclusion
$\k[\chi+\m_\k^\perp]^{\mathcal{M}_\k}\hookrightarrow
\k[\chi+\m_\k^\perp]$  is a single $\mathcal{M}_\k$-orbit. As the
maximal spectrum of $Z_p(\widetilde{\a}_\k)$ is isomorphic to
$(\chi+\m_\k^\perp)^{(1)}$ as $\mathcal{M}_\k$-varieties by our
earlier remarks, each fibre of the categorical quotient
$$\alpha\colon\,{\rm Specm}\,Z_p(\widetilde{\a}_\k)\longrightarrow\,
\big({\rm Specm}\,Z_p(\widetilde{\a}_\k)\big)/\!\!/\mathcal{M}_\k$$
is a single $\mathcal{M}_\k$-orbit as well. Now let $\Psi_i$,
$i=1,2$, be two linear functions as above, denote by $\psi_i$ the
restriction of $\Psi_i$ to $\a_\k$, and consider the corresponding
representations
$\widehat{\tau}_{\k,\,\psi_i}\colon\,\widehat{U}(\g_\k,e)\to \End
M_\k$. Since $\widehat{\tau}_{\k,\,\psi_1}$ and
$\widehat{\tau}_{\k,\,\psi_2}$ agree on
$Z_p(\widetilde{\a}_\k)^{\mathcal{M}_\k}\subset U(\g_\k,e)$, it must
be that $\alpha(\Psi_1)=\alpha(\Psi_2)$. But then $\Psi_1$ and
$\Psi_2$ are in the same $\mathcal{M}_\k$-orbit, as claimed.

\smallskip

The above discussion in conjunction with [\cite{P10}, Lemma~3.2]
also yields that for $p\gg 0$ the central subalgebra
$Z_p(\widetilde{\a})^{\mathcal{M}_\k}\cong
\k[(\chi+\m_\k^\perp)^{(1)}]^{\mathcal{M}_\k}$ of $U(\g_\k,e)$ is
isomorphic to the function algebra on the Frobenius twist of the
Slodowy slice $\mathbb{S}_\chi:=\chi+\kappa({\rm Ker}\,\ad f)$,
where $\kappa\colon\,\g_\k\stackrel{\sim}{\longrightarrow}\g_\k^*$
is the Killing isomorphism associated with the bilinear form
$(\,\cdot\,,\,\cdot\,)$. Together with [\cite{P10}, Thm.~2.1] this
shows that the relationship between $U_\eta(\g_\k,e)$ and
$U(\g_\k,e)$ is very similar to that between $U_\eta(\g_\k)$ and
$U(\g_\k)$. More precisely, one embeds $\k[(\mathbb{S}_\chi)^{(1)}]$
into $U(\g_\k,e)$ as an analogue of the $p$-centre $Z_p(\g_\k)$ (so
that $U(\g_\k,e)$ is a free $\k[(\mathbb{S}_\chi)^{(1)}]$-module of
rank $p^r$) and then obtains $U_\eta(\g_\k,e)$ from $U(\g_\k,e)$ by
tensoring the latter over $\k[(\mathbb{S}_\chi)^{(1)}]$ by a
suitable one-dimensional representation of
$\k[(\mathbb{S}_\chi)^{(1)}]$.
\end{rem}
\subsection{}\label{2.7}
Put $\I_A:={\rm Ann}_{U(\g_A)}\,\widetilde{M}_A$ and denote by ${\rm
gr}(\I_A)$ the corresponding graded ideal of $S(\g_A)$. Define
$R:=U(\g)/\I$, ${\rm gr}(R):=S(\g)/{\rm gr}(\I)$,
$R_A:=U(\g_A)/\I_A$, and ${\rm gr}(R_A)=S(\g_A)/{\rm gr}(\I_A)$.
Clearly, ${\rm gr}(R_A)=\bigoplus_{n\ge 0}({\rm gr}(R_A))(n)$ is a
finitely generated graded $A$-algebra and each $({\rm gr}(R_A))(n)$
is a finitely generated $A$-module. Also, $A$ is a commutative
Noetherian domain. If $b\in A\setminus\{0\}$, then ${\rm
gr}(\I_{A[b^{-1}]})={\rm gr}(\I_A)\otimes_A A[b^{-1}]$ and
\begin{eqnarray*}
{\rm gr}(R_{A[b^{-1}]})&=&S(\g_{A[b^{-1}]})/{\rm
gr}(\I_{A[b^{-1}]})\cong \big(S(\g_A)\otimes_A
A[b^{-1}]\big)/\big({\rm gr}(\I_A)\otimes_A A[b^{-1}]\big)\\
&\cong& {\rm gr}(R_A)\otimes_A A[b^{-1}];
\end{eqnarray*}
see [\cite{Bo}, Ch.~II, 2.4], for example. Since ${\rm
gr}(R)=\bigoplus_{n\ge 0}\,({\rm gr}(R))(n)$ is a graded Noetherian
algebra of Krull dimension $2d(e)=\dim\O$ with $({\rm
gr}(R))(0)=\CC$, we have that $2d(e)=\dim {\rm
gr}(R)=1+\deg\,P_R(t)$, where $P_{{\rm gr}(R)}(t)$ is the Hilbert
polynomial of ${\rm gr}(R)$; see [\cite{Eis}, Corollary~13.7].

\smallskip

Denote by $F$ the quotient field of $A$. Since ${\rm gr}(R_F):={\rm
gr}(R_A)\otimes_A F$ is a finitely generated algebra over a field,
the Noether Normalisation Theorem says that there exist homogeneous,
algebraically independent $y_1,\ldots, y_{2d(e)}\in{\rm gr}(R)_F$,
such that ${\rm gr}(R_F)$ is a finitely generated module over its
graded polynomial subalgebra $F[y_1,\ldots, y_{2d(e)}]$; see
[\cite{Eis}, Thm.~13.3]. Let $v_1,\ldots, v_D$ be a generating set
of the $F[y_1,\ldots, y_{2d(e)}]$-module ${\rm gr}(R_F)$ and let
$r_1,\ldots, r_{N}$ be a generating set of the $A$-algebra ${\rm
gr}(R_A)$. Then
\begin{eqnarray*}
v_i\cdot v_j&=&\textstyle{\sum}_{k=1}^D \,p_{i,j}^k(y_1,\ldots,
y_{d(e)})v_k\ \ \qquad(1\le
i,j\le D)\\
r_i&=&\textstyle{\sum}_{j=1}^{D}\,q_{i,j}(y_1,\ldots, y_{d(e)}) v_j\
\ \,\qquad(1\le i\le N)
\end{eqnarray*} for some polynomials $p_{i,j}^k,\, q_{i,j}\in
F[X_1,\ldots, X_{2d(e)}].$ The algebra ${\rm gr}(R_A)$ contains an
$F$-basis of ${\rm gr}(R_F)$. The coordinate vectors of the $r_i$'s,
$y_i$'s and $v_i$'s relative to this basis and the coefficients of
the polynomials $q_{i,j}$ and $p_{i,j}^k$ involve only finitely many
scalars in $F$. Replacing $A$ by $A[b^{-1}]$ for a suitable $0\ne
b\in A$ if necessary, we may assume that all $y_i$ and $v_i$ are in
${\rm gr}(R_A)$ and all $p_{i,j}^k$ and $q_{i,j}$ are in
$A[X_1,\ldots,X_{2d(e)}]$. In conjunction with our earlier remarks
this shows that no generality will be lost by assuming that
\begin{equation}\label{f-gen}{\rm gr}(R_A)\,=\,A[y_1,\ldots,
y_{2d(e)}]v_1+\cdots+A[y_1,\ldots,y_{2d(e)}]v_D
\end{equation} is a finitely generated module over the polynomial
algebra $A[y_1,\ldots, y_{2d(e)}]$.
\smallskip

Since ${\rm gr}(R_A)$ is a finitely generated $A[y_1,\ldots,
y_{d(e)}]$-module and $A$ is a Noetherian domain, a graded version
of the Generic Freeness Lemma shows that there exists a nonzero
element $a_1\in A$ such that each  $\big({\rm
gr}(R_A)(n)\big)[a_1^{-1}]$ is a free $A[a_1^{-1}]$-module of finite
rank; see (the proof of) Theorem~14.4 in [\cite{Eis}]. Since
$\big({\rm gr}(R_A)(n)\big)[a_1^{-1}]\cong \big({\rm gr}
(R_{A[a_1^{-1}]})\big)(n)$ for all $n$ by our earlier remarks, we
see that there exists an admissible ring $A\subset \CC$ such that
all graded components of ${\rm gr}(R_A)$ are free $A$-modules of
finite rank.

\smallskip

Since $S(\g_A)$ is a finitely generated $A$-algebra, we can also
apply the proof of Theorem~14.4 in [\cite{Eis}] to the graded ideal
${\rm gr}(\I_A)$ of $S(\g_A)$ to deduce that there exists a nonzero
$a_2\in A$ such that all graded components of $\big({\rm
gr}(\I_A)\big)[a_2^{-1}]$ are free $A[a_2^{-1}]$-modules of finite
rank. As $\big({\rm gr}(\I_A)\big)[a_2^{-1}]\cong{\rm
gr}(\I_{A[a_2^{-1}]})$ by [\cite{Bo}, Ch.~II, 2.4], we may (and we
will) assume that all graded components of  ${\rm gr}(\I_A)$ are
free $A$-modules of finite rank. Using the standard filtered-graded
techniques we now obtain that the $A$-modules $\I_A$ and $R_A$ are
free as well.

\subsection{}\label{2.8}
Note that $\widetilde{M}_F=\widetilde{M}_A\otimes_AF$ is a module
over the split Lie algebra $\g_F$. Since $\widetilde{M}\cong
\widetilde{M}_F\otimes_F\CC$, each subspace $\I\cap U_k(\g)$ is
defined over $F$ (here $U_k(\g)$ stands for the $k$th component of
the canonical filtration of $U(\g)$).  Since the algebra $U(\g)$ is
Noetherian, the ideal $\I$ is generated by its $F$-subspace
 $\I_{F,\,N}:=\,U_{N}(\g_F)\cap\I$. Since $\I$ is a
two-sided ideal of $U(\g)$, all subspaces $\I\cap U_k(\g)$ are
invariant under the adjoint action of $G$ on $U(\g)$. Hence the
$F$-subspace $\I_{F,\,N}$ is invariant under the adjoint action of
the distribution algebra $U_F:=U_\Z\otimes_\Z F$. Since
$\h_F:=\h\cap\g_F$ is a split Cartan subalgebra of $\g_F$, the
adjoint $\g_F$-module $\I_{F,\,N}$ decomposes into a finite direct
sum of absolutely irreducible $\g_F$-modules with integral dominant
highest weights. Therefore, the $\g_F$-module $\I_{F,\,N}$ possesses
a $\Z$-form invariant under the adjoint action of the Kostant
$\Z$-form $U_{\mathbb Z}$; we call it $\I_{\Z,\,N}$.

\smallskip

Let $\{u_i\,|\,\,i\in I\}$ be any basis of the free $\Z$-module
$\I_{\Z,\,N}$. Expressing the $u_i$ via the PBW basis of $U(\g_F)$
associated with the Chevalley basis $\mathcal B$ involves only
finitely many scalars in $F$. Enlarging $A$ further if need be we
may assume that all $u_i$ are in $U(\g_A)$ and hence that the ideal
$\I_A$ of $U(\g_A)$ is invariant under the natural action of the
Hopf $\Z$-algebra $U_\Z$ (one should keep in mind here that the
$\Z$-algebra $U_\Z$ is generated by all $e_\gamma^n/n!$ with
$\gamma\in\Phi$ and $n\in\Z_+$). Thus, from now on we may assume
that for any maximal ideal $\P$ of $A$ the two-sided ideal
$\I_\k:=\I_A\otimes_A\k_\P$ of $U(\g_\k)$ is stable under the
adjoint action of the simple algebraic $\k$-group $G_\k$ with
hyperalgebra $U_\k=U_\Z\otimes_\Z\k$.

\section{\bf Introducing certain finite subsets of regular elements in
$R$}
\subsection{}\label{3.1}
Let ${\mathcal B}=\{g_1,\ldots, g_n\}$ be our Chevalley basis of
$\g_\Z$ and identify ${\mathcal B}$ with its image in
$R=U(\g)/\mathcal{I}$. Denote by $R_k$ the $k$th component of the
filtration of $R$ induced by the canonical filtration of $U(\g)$ and
let $S$ be the Ore set of all regular elements in $R$. Since
${\mathcal Q}(R)=S^{-1}R\cong {\rm Mat}_{l'}({\mathcal D}_M)$, where
$l'={\rm rk}(R)$, there exists a unital subalgebra $\mathfrak C$ in
$\mathcal{Q}(R)$ isomorphic to ${\rm Mat}_{l'}(\CC)$ and such that
${\mathcal Q}(R)\cong \mathfrak{C}\otimes \mathfrak{D}$, where
$\mathfrak{D}\cong {\mathcal D}_M$ is the centraliser of
$\mathfrak{C}$ in $\mathcal{Q}(R)$.

\smallskip

In this section, we are going to describe an algorithmic procedure
that will produce at the end certain finite subsets $X\subset R$ and
$Y\subset S$. The subalgebra of ${\mathcal Q}(R)$ generated by $X$
and by $Y\cup Y^{-1}$ will contain the Chevalley basis
$\mathcal{B}$, a fixed set of matrix units of $\mathfrak{C}$ and a
generating set of the skew-field $\mathfrak{C}$. The actual form of
the elements in $X$ and $Y$ will be of no importance for us, but in
the next section we will rely on the fact that our procedure is
finite and each of its steps is reversible.

\smallskip

Fix a set $\{e_{ij}\,|\,\,1\le i,j\le l'\}$ of matrix units in
$\mathfrak{C}$, so that
\begin{eqnarray}\label{a}
e_{ij}e_{tk}&=&\delta_{jt}e_{ik}\qquad \qquad (1\le i,j,t,k\le l');\label{a}\\
\textstyle{\sum}_{i=1}^{l'}\,e_{ii}&=&1.\label{b}
\end{eqnarray}
There exist $s_{ij}, s'_{ij}\in S$ and $E_{ij}, E_{ij}'\in R$ such
that
\begin{eqnarray}\label{0}
s_{ij}^{-1}E_{ij}\,=\,e_{ij}\,=\,E_{ij}'(s_{ij}')^{-1}.
\end{eqnarray}
 Then in $R$
we have the following relations
\begin{eqnarray}\label{00}
E_{ij}s_{ij}'\,=\,s_{ij}E_{ij}'\qquad\qquad\quad(1\le i,j\le l').
\end{eqnarray}
As $\mathcal{Q}(R)=\mathfrak{C}\otimes \mathfrak{D}$, there exist
$c_{ij}^k\in \mathcal{Q}(R)$, where $1\le k\le n$, such that
\begin{eqnarray}
g_k&=&\textstyle{\sum}_{i,j}\,e_{ij}c_{ij}^k\qquad\qquad(1\le k\le n);\label{d}\\
c_{ij}^ke_{th}&=&e_{th}c_{ij}^k\qquad\qquad\qquad(1\le i,j,t,h\le
l';\ 1\le k\le n).\label{e}
\end{eqnarray}
For each $k\le l'$ we can find $a^k_{ij}\in S$ and $C_{ij}^k\in R$
such that $c_{ij}^k=(a_{ij}^k)^{-1}C_{ij}^k$. Since $S$ is an Ore
set, there are $r_{ij,\,tk}, r^k_{ij,\,th}, a^k_{ij,\,th}\in S$ and
$E_{ij,\,tk}, E^k_{ij,\,th}, C^k_{ij,\,th}\in R$ such that
\begin{eqnarray}
r_{ij,\,tk}E_{ij}&=&E_{ij,\,tk}s_{tk}\qquad\qquad\ \, (1\le
i,j,t,k\le
l');\label{f}\\
r^k_{ij,\,th}C_{ij}^k&=&E^k_{ij,\,th}s_{th}\qquad\qquad\ \, (1\le
i,j,t,h\le l',\ 1\le k\le l');\label{g}\\
C_{ij}^ka^k_{ij,\,th}&=&a_{ij}^ks_{th}'C^k_{ij,\,th}\qquad\quad\
\,(1\le i,j,t,h\le l',\ 1\le k\le l').\label{h}
\end{eqnarray}
Since
$s_{ij}^{-1}E_{ij}s^{-1}_{tk}E_{tk}\,=\,\delta_{jt}E'_{ik}(s'_{ik})^{-1}$
by (\ref{a}) and (\ref{00}), applying (\ref{f}) we obtain that
$s_{ij}^{-1}r_{ij,\,tk}^{-1}E_{ij,
tk}E_{tk}\,=\,\delta_{jt}E'_{ik}(s'_{ik})^{-1}$. This yields
\begin{eqnarray}\label{i}
E_{ij,\,tk}E_{ik}s'_{ik}\,=\,\delta_{jt}r_{ij,\,tk}s_{ij}E'_{ik}\qquad\qquad\
(1\le i,j,t,k\le l').
\end{eqnarray}
Similarly, since
$(a_{ij}^k)^{-1}C_{ij}^ks_{th}^{-1}E_{th}\,=\,E_{th}'(s'_{th})^{-1}(a_{ij}^k)^{-1}C_{ij}^k$
by (\ref{e}) and (\ref{00}), applying (\ref{g}) and (\ref{h}) yields
$(a_{ij}^k)^{-1}(r^k_{ij,\,th})^{-1}E_{ij,\,th}^kE_{th}\,=\,E_{th}'C_{ij,\,th}^k(a_{ij,\,th}^k)^{-1}$.
We thus get
\begin{eqnarray}\label{j}
\quad\
E_{ij,\,th}^kE_{th}a_{ij,th}^k\,=\,r^k_{ij,\,th}a_{ij}^kE_{th}'C_{ij,\,th}^k\qquad\quad(1\le
i,j,t,h\le l',\ 1\le k\le l').
\end{eqnarray}

Recall that $1\,=\,\sum_{i=1}^{l'}e_{ii}\,=\,
\sum_{i=1}^{l'}s_{ii}^{-1}E_{ii}.$ Multiplying both sides by
$s_{11}$ on the left we get
\begin{eqnarray}\label{k}
s_{11}\,=\,E_{11}+\textstyle{\sum}_{i=2}^{l'}\,s_{11}s_{ii}^{-1}E_{ii}.
\end{eqnarray}
There exist $s_{1,2}\in S$ and $q_{2}\in R$ such that
$s_{1,2}s_{11}=q_{2}s_{22}.$ Multiplying both sides of (\ref{k}) by
$s_{1,2}$ on the left we then obtain
\begin{eqnarray}\label{l}
\qquad s_{1,2}s_{11}\,=\,s_{1,2}E_{11}+q_{2}E_{22}+
\textstyle{\sum}_{i=3}^{l'}\,s_{1,2}s_{11}s_{ii}^{-1}E_{ii}.
\end{eqnarray}
For $3\le k\le l'$, we select (recursively) some $s_{1,\ldots,k}\in
S$ and $q_{k}\in R$ such that
\begin{eqnarray}\label{m}
\textstyle{\prod}_{i=1}^k s_{1,\ldots, k-i+1}\,=\,q_{k}s_{kk}.
\end{eqnarray}
For convenience, we set $q_{1}=1$. At the end of the process started
with (\ref{k}) and (\ref{l}) we get rid of all denominators and
arrive at the relation
\begin{eqnarray}\label{n}
\prod_{k=1}^{l'}s_{1,\ldots,l'-k+1}\,=\,
\sum_{k=1}^{l'}\Big(\prod_{i=1}^{l'-k}s_{1,\ldots,l'-i+1}\Big)q_{k}E_{kk}
\end{eqnarray}
which holds in $R$.

\smallskip

Let $p(1),\ldots, p(l'^2)$ be all elements in the lexicographically
ordered set $\{(i,j)\,|\,\,1\le i,j\le l'^2\}$ and denote by
$e_{p(k)}$, $E_{p(k)}$ and $s_{p(k)}$ the corresponding elements in
$R$. Since $e_{ij}$ commutes with $c_{ij}^k$ we can rewrite
(\ref{d}) as
\begin{eqnarray}\label{o}
g_k\,=\,\textstyle{\sum}_{i,j}\,(a_{ij}^k)^{-1}C_{ij}^ks_{ij}^{-1}E_{ij}\qquad\qquad(1\le
k\le n).
\end{eqnarray}
For $1\le k\le l'$, there exist $D_{ij}^k\in R$ and $s_{ij}^k\in S$
such that
\begin{eqnarray}\label{p}\qquad
D_{ij}^ks_{ij}\,=\,s_{ij}^kC_{ij}^k\ \qquad\quad(1\le i,j\le l').
\end{eqnarray}
Then, setting $T_{ij}^k\,:=\,D_{ij}^kE_{ij}$ and
$s_{ij;\,k}\,:=\,s_{ij}^ka_{ij}^k$, we can rewrite (\ref{o}) as
follows:
\begin{eqnarray}\label{r}
g_k\,=\,\textstyle{\sum}_{i,j}\,s_{ij;\,k}^{-1}T_{ij}^k\,=\,
\textstyle{\sum}_{i=1}^{l'^2}s_{p(i);\,k}^{-1}T_{p(i)}^k\qquad\qquad(1\le
k\le n).
\end{eqnarray}
Multiplying both sides of (\ref{r}) by $s_{p(1);\,k}$ on the left we
get
\begin{eqnarray}\label{q}
s_{p(1);\,k}\cdot
g_k\,=\,T^k_{p(1)}+\textstyle{\sum}_{i=2}^{l'^2}\,s_{p(1);\,k}s_{p(i);\,k}^{-1}T_{p(i)}.
\end{eqnarray}
There are $s_{p(1),p(2);\,k}\in S$ and $q_{p(2)}^k\in R$ such that
$s_{p(1),p(2);\,k}s_{p(1);\,k}=q_{p(2)}^k s_{p(2);\,k}.$ Multiplying
both sides of (\ref{q}) by $s_{p(1),p(2);\,k}$ on the left we get $$
(s_{p(1),p(2);\,k}s_{p(1);\,k})g_k\,=\,s_{p(1),p(2);\,k}T_{p(1)}^k+q^k_{p(2)}T^k_{p(2)}+
\textstyle{\sum}_{i=3}^{l'^2}\,s_{p(1),p(2);\,k}s_{p(1);\,k}s_{p(i);\,k}^{-1}T^k_{p(i)}.
$$
For $3\le j\le l'^2$, we choose (recursively) some
$s_{p(1),\ldots,p(j);\,k}\in S$ and $q^k_{p(j)}\in R$ such that
\begin{eqnarray}\label{r'}
\textstyle{\prod}_{i=1}^j s_{p(1),\ldots,
p(j-i+1);\,k}\,=\,q^k_{p(j)}s_{p(j);\,k},
\end{eqnarray}
and set $q^k_{p(1)}=1$. As before, at the end of the process just
started we arrive at the relations
\begin{eqnarray}\label{s}
\Big(\prod_{j=1}^{l'^2}s_{p(1),\ldots,p(l'^2-j+1);\,k}\Big)g_k\,=\,
\sum_{j=1}^{l'^2}\Big(\prod_{i=1}^{l'^2-j}s_{p(1),\ldots,p(l'^2-i+1);\,k}\Big)q^k_{p(j)}T^k_{p(j)}
\end{eqnarray}
which hold in $R$, where $1\le k\le n$.

\smallskip

We now denote by $X$ the set of all elements
\begin{eqnarray*}
&&E_{ij}, E_{ij}',C_{ij}^k, E_{ij,tk}, E_{ij,th}^k, C_{ij,th}^k,
q_k, D_{ij}^k, T_{ij}^k, q_{p(j);\,k}
\end{eqnarray*}
 and by $Y$ the set of all elements
\begin{eqnarray*}
&&s_{ij}, s'_{ij},a_{ij}^k, r_{ij,tk}, r_{ij,th}^k, a_{ij,th}^k,
s_{1,\ldots,k}, s_{ij}^k, s_{ij;\,k},
s_{p(1),\ldots,p(j);\,k}.\end{eqnarray*}
\subsection{}\label{3.2} In this subsection we assume that $\mathfrak{D}$ is a Weyl
skew-field, more precisely, $\mathfrak{D}\cong \mathcal{Q}({\mathbf
A}_{d(e)}(\CC))$. We follow closely the exposition in [\cite{P10'},
Sect.~2] and adopt (with some minor modifications) the notation
introduced there.

\smallskip

Set $d:=d(e)$. If a pair $(a,b)\in\{(i,j)\,|\,\,1\le i,j\le l'\}$
occupies the $k$th place in our lexicographical ordering, then we
write $c_{p(k)}^s$, $a_{p(k)}^s$ and $C_{p(k)}^s$ for $c_{ab}^s$,
$a_{ab}^s$ and $C_{ab}^s$, respectively. There exist $w_1,\ldots,
w_{2d}\in \mathfrak{D}$ such that
\begin{eqnarray}
[w_i,w_j]\,=\,[w_{d+i},w_{d+j}]&=&0\,\,\, \qquad \quad\quad(1\le i,j\le d);\label{1}\\
\,[w_i,w_{d+j}]&=&\delta_{i,j}\,\, \qquad \,\, \quad(1\le i,j\le d);\label{2}\\
Q^s_{p(k)}\cdot c^s_{p(k)}&=&P^s_{p(k)},\qquad \, \,\,(1\le k\le
l'^2;\ 1\le s\le n)\label{3}
\end{eqnarray}
for some nonzero polynomials $P^s_{p(k)},Q^s_{p(k)}$ in $w_1,\ldots,
w_{2d}$ with coefficients in $\CC$. (One should keep in mind here
that the monomials $w_1^{a_1}w_2^{a_2}\cdots w_{2d}^{a_{2d}}$ with
$a_i\in\Z_+$ form a basis of the $\CC$-subalgebra of $D$ generated
by $w_1,\ldots,w_{2d}$.)

\smallskip

Since every nonzero element of $\mathfrak{D}$ is regular in
$\mathcal{Q}(R)$, there exist $Q^s_{1;\,p(k)}, Q^s_{2;\,p(k)}\in S$
such that
\begin{eqnarray}\label{qq}
Q^s_{p(k)}Q^s_{1;\,p(k)}\,=\,Q^s_{2;\,p(k)}\qquad \, \,\,(1\le k\le
l'^2;\ 1\le s\le n).
\end{eqnarray}
Since $w_i=v_i^{-1}u_i$ for some elements $v_i\in S$ and $u_i\in R$,
we can rewrite (\ref{1}) and (\ref{2}) as follows:
\begin{eqnarray}
v_i^{-1}u_i\cdot v_j^{-1}u_j&=&v_j^{-1}u_j\cdot v_i^{-1}u_i;\label{1'}\\
v_{d+i}^{-1}u_{d+i}\cdot v_{d+j}^{-1}u_{d+j}&=&v_{d+j}^{-1}u_{d+j}\cdot v_{d+i}^{-1}u_{d+i};\label{2'}\\
v_{i}^{-1}u_{i}\cdot v_{d+j}^{-1}u_{d+j}-v_{d+j}^{-1}u_{d+j}\cdot
v_{i}^{-1}u_{i}&= &\delta_{i,j}\qquad\qquad \quad\quad \ (1\le
i,j\le d). \label{3'}
\end{eqnarray}
As $S$ is an Ore set, there are $v_{i,j}\in S$ and $u_{i,j}\in R$
such that
\begin{equation}\label{uv}
v_{i,j}u_i=u_{i,j}v_j\qquad\qquad\quad(1\le i,j\le 2d).
\end{equation}
Thus we can rewrite (\ref{1'}), (\ref{2'}) and (\ref{3'}) in the
form
\begin{eqnarray}
\quad v_i^{-1}v_{i,j}^{-1}\cdot
u_{i,j}u_j&=&v_j^{-1}v_{j,i}^{-1}\cdot u_{j,i}u_i
\qquad(1\le i,j\le d\mbox{ or } d\le i,j\le 2d)\label{1''}\\
\quad v_{i}^{-1}v_{i,d+j}^{-1}\cdot
u_{i,d+j}u_{d+j}&=&\delta_{ij}+v_{d+j}^{-1}v_{d+j,i}^{-1}\cdot
u_{d+j,i}u_{d+i}\qquad(1\le i,j\le d).\label{2''}
\end{eqnarray}
There exist $b_{i,j}\in S$ and  $b'_{i,j}\in R$ such that
\begin{equation}\label{10}
b_{i,j}v_{i,j}v_i=b'_{i,j}v_{j,i}v_j\qquad\qquad\quad(1\le i,j\le
2d).
\end{equation}
Since $v_{i,j}v_i(v_{j,i}v_j)^{-1}\,=\,b_{i,j}^{-1}b'_{i,j},$ we see
that (\ref{1''}) and (\ref{2''}) give rise to the relations
\begin{eqnarray}
\quad b_{i,j}u_{i,j}u_j&=&b'_{i,j}u_{j,i}u_i
\qquad(1\le i,j\le d\mbox{ or } d\le i,j\le 2d)\label{1'''}\\
\qquad b_{i,d+j}u_{i,d+j}u_{d+j}&=&\delta_{ij}b_{i,d+j}v_{i,d+j}v_i+
b'_{i,d+j} u_{d+j,i}u_{i}\qquad(1\le i,j\le d)\label{2'''}
\end{eqnarray}
which hold in $R$.

\smallskip

For an $m$-tuple ${\bf i}=(i(1),i(2),\ldots,i(m))$ with $1\le
i(1)\le i(2)\le \cdots\le i(m)\le 2d$ and $m\ge 3$ we select
(recursively) some $u_{i(1),\ldots, i(k)}\in R$ and
$v_{i(1),\ldots,i(k)}\in S$, where $3\le k\le m$, such that
\begin{equation}\label{13}
v_{i(1),\ldots,
i(k)}u_{i(1),\ldots,i(k-1)}u_{i(k-1)}=u_{i(1),\ldots, i(k)}v_{i(k)}.
\end{equation}
Write $w^{\mathbf i}:=w_{i(1)}\cdot w_{i(2)}\cdot\ldots\cdot
w_{i(m)} =\prod_{k=1}^m\,v_{i(k)}^{-1}u_{i(k)}.$ Then
\begin{eqnarray*}
w^{\mathbf i}&=&v_{i(1)}^{-1}u_{i(1)}\cdot
v_{i(2)}^{-1}u_{i(2)}\cdot\prod_{k=3}^m\,v_{i(k)}^{-1}u_{i(k)}\\
&=&v_{i(1)}^{-1}v_{i(1),i(2)}^{-1}u_{i(1),i(2)}u_{i(2)}\cdot
v_{i(3)}^{-1}u_{i(3)}\cdot
\prod_{k=4}^m\,v_{i(k)}^{-1}u_{i(k)}\\
&=&v_{i(1)}^{-1}v_{i(1),i(2)}^{-1}v_{i(1),i(2),i(3)}^{-1}
u_{i(1),i(2),i(3)}u_{i(3)}\cdot
\prod_{k=4}^m\,v_{i(k)}^{-1}u_{i(k)}\\
&=&\cdots\,=\,\Big(\prod_{k=1}^m\,v_{i(1),\ldots, i(m-k+1)}
\Big)^{-1}\cdot u_{i(1),\ldots, i(m)}u_{i(m)}.
\end{eqnarray*}
Then we set $v_{\bf i}:={\prod}_{k=1}^m\,v_{i(1),\ldots, i(m-k+1)}$,
an element of $S$, and $u_{\bf i}:=u_{i(1),\ldots, i(m)}u_{i(m)}$,
an element of $R$.

\smallskip

Let $\{{\bf i}(1),\ldots,{\bf i}(N)\}$ be the set of all tuples as
above with $\sum_{\ell=1}^Ni(\ell)\le \Delta$, where
$\Delta=\max\big\{\deg P^s_{p(k)},\,\deg Q^s_{p(k)}\,|\,\,1\le k\le
l'^2,\,1\le s\le n\big\}$. Clearly,
$P^s_{p(k)}=\sum_{j=1}^N\,\lambda^s_{j,k}w^{{\bf i}(j)}$ and
$Q^s_{p(k)}=\sum_{j=1}^N\,\mu^s_{j,k}w^{{\bf i}(j)}$ for some
$\lambda^s_{j,k},\mu^s_{j,k}\in \CC$, where $1\le k\le l'^2$ and
$1\le s\le n$. By the above, we have that
$P^s_{p(k)}=\sum_{j=1}^N\,\lambda^s_{j,k}\,v_{{\bf i}(j)}^{-1}
u_{{\bf i}(j)}$ and $Q^s_{p(k)}=\sum_{i=1}^N\, \mu^s_{j,k}\,v_{{\bf
i}(j)}^{-1}u_{{\bf i}(j)}$.

Set $v_{{\bf i}(j)}(0):=v_{{\bf i}(j)}$ and $u_{{\bf
i}(j)}(0)=u_{{\bf i}(j)}$. For each pair $(j,t)$ of positive
integers satisfying $N\ge j>t>0$ we select (recursively) some
$v_{{\bf i}(j)}(t)\in S$ and $u_{{\bf i}(j)}(t)\in R$ such that
\begin{eqnarray}\label{06}
v_{{\bf i}(j)}(t)v_{{\bf i}(t)}(t-1)\,=\,u_{{\bf i}(j)}(t)v_{{\bf
i}(j)}(t-1).
\end{eqnarray}
Multiplying both sides of (\ref{3}) by $v_{{\bf i}(1)}$ on the left
and applying (\ref{06}) with $t=1$ we obtain that
\begin{eqnarray*}
v_{{\bf i}(1)}P_{p(k)}^s&=&\lambda^s_{1,k}u_{{\bf i}(1)}
+\sum_{j=2}^N\lambda^s_{j,k}v_{{\bf i}(1)}v_{{\bf i}(j)}^{-1}
u_{{\bf i}(j)}\\
&=&\lambda^s_{1,k}u_{{\bf i}(1)} +\sum_{j=2}^N\lambda^s_{j,k}v_{{\bf
i}(j)}(1)^{-1}u_{{\bf i}(j)}(1) u_{{\bf i}(j)}.
\end{eqnarray*}
Multiplying both sides of this equality by $v_{{\bf i}(2)}(1)$ on
the left  and applying (\ref{06}) with $t=2$ we get
\begin{eqnarray*}
v_{{\bf i}(2)}(1)v_{{\bf i}(1)}P_{p(k)}^s&=&\lambda^s_{1,k}v_{{\bf
i}(2)}(1)u_{{\bf i}(1)} +\lambda^s_{2,k}u_{{\bf i}(2)}(1)u_{{\bf
i}(1)}\\ &+&\sum_{j=3}^N\lambda^s_{j,k}v_{{\bf i}(j)}(2)^{-1}u_{{\bf
i}(j)}(2)u_{{\bf i}(j)}(1) u_{{\bf i}(j)}.
\end{eqnarray*}
Repeating this process $N$ times we arrive at the relation
\begin{equation}
\ \ \Big(\prod_{\ell=1}^{N}v_{{\bf
i}(N-\ell+1)}\Big)P_{p(k)}^s\,=\,\,\sum_{j=1}^N\lambda^s_{j,k}\cdot\Big(\prod_{\ell=1}^{N-j}v_{{\bf
i}(N-\ell+1)}(N-\ell)\cdot\prod_{\ell=1}^{j} u_{{\bf
i}(j-\ell+1)}(j-\ell)\Big) \label{18}
\end{equation}
which holds in $R$  (at the $\ell$-th step of the process we
multiply the preceding equality by $v_{{\bf i}(\ell)}(\ell-1)$ on
the left and then apply (\ref{06}) with $t=\ell$). Similarly, we
have that
\begin{equation}
\ \ \Big(\prod_{\ell=1}^{N}v_{{\bf
i}(N-\ell+1)}\Big)Q_{p(k)}^s\,=\,\,\sum_{j=1}^N\mu^s_{j,k}\cdot\Big(\prod_{\ell=1}^{N-j}v_{{\bf
i}(N-\ell+1)}(N-\ell)\cdot\prod_{\ell=1}^{j} u_{{\bf
i}(j-\ell+1)}(j-\ell)\Big). \label{18'}
\end{equation}
We denote the left-hand sides of (\ref{18}) and (\ref{18'}) by
$\widetilde{P}_{p(k)}^s$ and $\widetilde{Q}_{p(k)}^s$, respectively,
and set $\tilde{v}:=\prod_{\ell=1}^{N}v_{{\bf i}(N-\ell+1)}$. Note
that $\tilde{v}\in S$. Then
\begin{equation}\label{18''}
 \tilde{v}^{-1}\widetilde{P}^s_{p(k)}\,=\,P_{p(k)}^s,\qquad
\tilde{v}^{-1}\widetilde{Q}^s_{p(k)}\,=\,Q_{p(k)}^s
\qquad\qquad(1\le k\le N;\ 1\le s\le l'^2).
\end{equation}
Now (\ref{3}) can be rewritten as
\begin{equation}\label{14}
\widetilde{Q}^s_{p(k)}(a_{p(k)}^s)^{-1}C_{p(k)}^s\,=\,
\widetilde{P}_{p(k)}^s\qquad\qquad (1\le k\le N;\ 1\le s\le l'^").
\end{equation}
Choosing $\tilde{a}_{p(k)}^s\in S$ and $\tilde{q}_{p(k)}^s\in R$
such that
\begin{eqnarray}\label{15}
\tilde{a}_{p(k)}^s\widetilde{Q}^s_{p(k)}\,=\,a_{p(k)}^s\tilde{q}_{p(k)}^s\qquad\qquad
(1\le k\le N;\ 1\le s\le l'^2)
\end{eqnarray}
we can rewrite (\ref{14}) as follows:
\begin{eqnarray}\label{16}
\tilde{q}^s_{p(k)}C_{p(k)}^s\,=\,\tilde{a}_{p(k)}^s\widetilde{P}_{p(k)}^s
\qquad\qquad(1\le k\le N;\ 1\le s\le l'^2).
\end{eqnarray}
This relation holds in $R$. In view of (\ref{qq}) we have that
\begin{equation*}
Q^s_{p(k)}\,=\,Q^s_{2;\,p(k)}(Q^s_{1;\,p(k)})^{-1}\qquad\qquad(1\le
k\le N;\ 1\le s\le l'^2).
\end{equation*}
Combining this with (\ref{18''}) we obtain
\begin{equation}\label{17}
\widetilde{Q}^s_{p(k)}Q^s_{1;\,p(k)}\,=\,\tilde{v}Q^s_{2;\,p(k)}\qquad\qquad
(1\le k\le N;\ 1\le s\le l'^2).
\end{equation}
This relation holds in $R$ as well.

\smallskip

Finally, in view of (\ref{3}) and (\ref{qq}) we can replace
(\ref{e}) by the following relation:
\begin{eqnarray}\label{d'}
e_{ij}w_t\,=\,w_te_{ij}\qquad\qquad\quad(1\le i,j\le l';\ 1\le t\le
2d).
\end{eqnarray}
Since $v_i^{-1}u_i=w_i=u_i'v_i'^{-1}$ for some $u_i'\in R$ and
$v_i'\in S$, the latter can be rewritten as
$$
s_{ij}^{-1}E_{ij}v_t^{-1}u_t\,=\,u_t'v_t'^{-1}s_{ij}^{-1}E_{ij}\qquad\qquad(1\le
i,j\le l';\ 1\le t\le 2d). $$ There exists $v_{ij;\,t},
b_{ij;\,t}\in S$ and $E_{ij;\,t}, D_{ij;\,t}\in R$ such that
\begin{eqnarray}
v_{ij;\,t}E_{ij}&=&E_{ij;\,t}v_t;\label{D1}\\
s_{ij}v_t'D_{ij;\,t}&=&E_{ij}b_{ij;\,t}\label{D2}
\end{eqnarray}
for all $1\le i,j\le l'$ and $1\le t\le 2d$. Then (\ref{d'}) gives
rise to the relations
\begin{eqnarray}\label{D''}
\qquad\qquad
E_{ij;\,t}u_tb_{ij;\,t}\,=\,v_{ij;\,t}s_{ij}u_t'D_{ij;\,t}\qquad\quad(1\le
i,j\le l';\ 1\le t\le 2d)
\end{eqnarray}
which hold in $R$.

\smallskip

We denote by $X$ the set of all elements
\begin{eqnarray*}
&&E_{ij}, E_{ij}',C_{ij}^k, E_{ij,tk}, E_{ij,th}^k, C_{ij,th}^k,
q_k, D_{ij}^k, T_{ij}^k, q_{p(j);\,k},\\&& u_i, u_i', u_{i,j},
b_{i,j}', u_{i(1),\ldots, i(k)}, u_{{\bf
i}(j)}(t),\tilde{q}_{p(k)}^s, E_{ij;\,t}, D_{ij;\,t}
\end{eqnarray*}
 and by $Y$ the set of all elements
\begin{eqnarray*}
&&s_{ij}, s'_{ij},a_{ij}^k, r_{ij,tk}, r_{ij,th}^k, a_{ij,th}^k,
s_{1,\ldots,k}, s_{ij}^k, s_{ij;\,k}, s_{p(1),\ldots,p(j);\,k},\\&&
Q_{1;\,p(k)}^s, Q_{2;\,p(k)}^s, v_i, v_i', v_{i,j}, b_{i,j},
v_{i(1),\ldots,i(k)}, v_{{\bf i}(j)}(t),\tilde{a}_{p(k)}^s,
v_{ij;\,t}, b_{ij;\,t}.\end{eqnarray*}
\subsection{}\label{3.3} Let $X\subset R$ and $Y\subset S$
be the finite subsets introduced in \ref{3.1} and \ref{3.2}.
Obviously, they lie in $R_m$ for some $m\gg 0$, hence involve only
finitely many scalars in $\CC$.  From now on we will always assume
that those scalars are in $A$ and hence $X\cup Y\subset R_A$. It
will be crucial for us in what follows to work with those admissible
rings $A$  for which the images of the elements of $Y$ in
$R_\k\,=\,(R_A/\mathfrak{P}R_A)\otimes_{A/\mathfrak{P}}\,\k$ remain
regular for all maximal ideals $\mathfrak{P}$ of $A$. Our next
result ensures that such admissible rings do exist.
\begin{lemma}\label{reg}
Let $s$ be a regular element of $R$ contained in $R_A$ and assume
that $A$ satisfies the conditions imposed in \ref{2.7}. Then there
exists an admissible extension $B$ of $A$ such that for every
$\mathfrak{P}\in{\rm Specm}\,B$ the element $s\otimes 1$ is regular
in $R_B\otimes_B\k_{\mathfrak{P}}\,\cong\,
(R_B/\mathfrak{P}R_B)\otimes_{B/\mathfrak{P}}\k$.
\end{lemma}
\begin{proof}
Since $s\cdot R_A$ is a right ideal of $R_A$, the graded $A$-module
$\gr(s\cdot R_A)$ is an ideal of the commutative Noetherian ring
$\gr(R_A)$. Hence $\gr(s\cdot R_A)$ is a finitely generated
$\gr(R_A)$-module. As $A$ is a Noetherian domain, applying
[\cite{Eis}, Thm.~14.4] shows that there is a nonzero  $a_1\in A$
such that each $\big(\gr(s\cdot R_A)(n)\big)[a_1^{-1}]$ is a free
$A[a_1^{-1}]$-module of finite rank. Since $\big(\gr(s\cdot
R_A)(n)\big)[a_1^{-1}]\cong \big(\gr(s\cdot
R_{A[a_1^{-1}]})\big)(n)$ for all $n$, we see that there exists an
admissible ring $\tilde{A}\subset \CC$ containing $A$ such that all
graded components of $\gr(s\cdot R_{\tilde{A}})$ are free
$\tilde{A}$-modules of finite rank. Since we can repeat this
argument with the left ideal $R_A\cdot s$ in place of $s\cdot R_A$,
it can be assumed, after enlarging $\tilde{A}$ possibly, that all
graded components of $\gr(R_{\tilde{A}}\cdot s)$ are free
$\tilde{A}$-modules of finite rank as well.

\smallskip

Since $\gr(R_A)$ is a finitely generated $A$-algebra, we can also
apply [\cite{Eis}, Thm.~14.4] to the graded ${\rm gr}(R_A)$-module
$\gr(R_A/s\cdot R_A)\cong \gr(R_A)/\gr(s\cdot R_A)$ to deduce that
there is a nonzero $a_2\in A$ such that all graded components of
$$\gr(R_A/s\cdot R_A)[a^{-1}_2]\,\cong\,\big(\gr(R_A)/\gr(s\cdot
R_A)\big)[a_2^{-1}]\,\cong\,\gr(R_{A[a_2^{-1}]})/\gr(s\cdot
R_{A[a_2^{-1}]}) $$ are free $A[a_2^{-1}]$-modules of finite rank.
Replacing $s\cdot R_A$ by $R_A\cdot s$ in this argument we observe
that the same applies to all graded components of
$\gr(R_{A[a_3^{-1}]})/\gr(R_{A[a_3^{-1}]}\cdot s)$ for a suitable
nonzero $a_3\in A$.

\smallskip

We conclude that there exists an admissible extension $B$ of $A$
such that all graded components of $\gr(s\cdot R_B)$, $\gr(R_B\cdot
s)$, $\gr(R_B)/\gr(s\cdot R_B)$ and $\gr(R_B)/\gr(R_B\cdot s)$ are
free $B$-modules of finite rank. Straightforward induction on
filtration degree now shows that the free $B$-modules $s\cdot
R_B\cong R_B$ and $R_B\cdot s\cong R_B$ are direct summands of
$R_B$. Let $R_B'$ and $R_B''$ be $B$-submodules of $R_B$ such that
$R_B=(s\cdot R_B)\oplus R_B'$ and $R_B=(R_B\cdot s)\oplus R_B''$.

\smallskip
We now take any maximal ideal $\mathfrak{P}$ of $B$, denote by
$\mathfrak{f}$ the finite field $B/\mathfrak{P}$, and write
$\bar{x}$ for the image of $x\in R_B$ in
$R_\k=(R_B/\mathfrak{P}R_B)\otimes_\mathfrak{f}\k$. Note that
$R_{\mathfrak f}\,:=\,R_B/\mathfrak{P}R_B$ is an $\mathfrak{f}$-form
of the $\k$-vector space $R_\k$. Suppose $\bar{s}\cdot \bar{u}=0$
for some $u\in R_B$. Then
\begin{eqnarray*}
s\cdot u&\in& (s\cdot R_B)\cap\mathfrak{P}R_B\,=\,(s\cdot
R_B)\cap\big(\mathfrak{P}(s\cdot R_B)\oplus
\mathfrak{P}R_B'\big)\\
&=&(s\cdot R_B)\cap\big(s\cdot \mathfrak{P}R_B)\oplus
\mathfrak{P}R_B'\big)\,=\,s\cdot\mathfrak{P}R_B.
\end{eqnarray*}
Therefore, $s\cdot u=s\cdot u'$ for some $u'\in \mathfrak{P}R_B$.
Since $s$ is a regular element of $R$ and $s\cdot(u-u')=0$, we
deduce that $u=u'\in \mathfrak{P}R_B$. This yields $\bar{u}=0$. If
$\bar{v}\cdot \bar{s}=0$ for some $v\in R_B$, then we use the
decomposition $R_B=(R_B\cdot s)\oplus R_B''$ and argue as before to
deduce that $\bar{v}=0$. Hence $\bar{s}$ is a regular element of
$R_{\mathfrak{f}}$.

\smallskip

Let $l_{\bar{s}}\colon\,R_\k\to R_\k$ and
$r_{\bar{s}}\colon\,R_\k\to R_\k$ denote the left and right
multiplication by $\bar{s}$, respectively. Denote by $(R_\k)_j$ the
$j$th component of the filtration of $R_\k$ induced by the canonical
filtration of $U(\g_\k)$ and set $(R_{\mathfrak f})_j:=(R_\k)_j\cap
R_{\mathfrak f}$. We know that $\bar{s}\in (R_{\mathfrak f})_\ell$
for some $\ell$, whereas the regularity of $\bar{s}$ in
$R_{\mathfrak f}$ yields that the $\mathfrak f$-linear maps
$l_{\bar{s}}\colon\,(R_{\mathfrak f})_j\to(R_{\mathfrak
f})_{j+\ell}$ and $r_{\bar{s}}\colon\,(R_{\mathfrak
f})_j\to(R_{\mathfrak f})_{j+\ell}$ are injective for all $j\in
\Z_+$. Standard linear algebra then shows that so are all
$\k$-linear maps $l_{\bar{s}}\colon\,(R_\k)_j\to(R_\k)_{j+\ell}$ and
$r_{\bar{s}}\colon\,(R_\k)_j\to(R_\k)_{j+\ell}$. In other words,
$\bar{s}$ is regular in $R_\k$ as claimed.
\end{proof}
\section{\bf  Proving the main results}
\subsection{}\label{4.1} From now on we assume that for every $s\in Y$ the
element $s\otimes 1$ is regular in $R_\k=(R_A/\P
R_A)\otimes_{\mathfrak f}\,\k$ for every $\P\in{\rm Specm}\,A$ (here
$\mathfrak{f}=A/\P$). Since $Y$ is a finite set, this is a valid
assumption thanks to Lemma~\ref{reg}. We also assume that our
admissible ring $A$ satisfies all requirements mentioned in Sect.~2.
The discussion in \ref{2.8} then shows that the simple algebraic
group $G_\k$ acts on $R_\k$ as algebra automorphisms and preserves
the filtration of $R_\k$ induced by the canonical filtration of
$U(\g_\k)$.

\smallskip

Since $U(\g_\k)$ is a finite module over its centre, so is its
homomorphic image $R_\k\,=\,\big(U(\g_A)/\I_A\big)\otimes_{\mathfrak
f}\,\k\,\cong\,U(\g_\k)/\I_\k$. Being a homomorphic image of
$U(\g_\k)$, the ring $R_\k$ is Noetherian and, moreover, an affine
PI-algebra over $\k$. Let $I_1,\ldots, I_\nu$ be the minimal primes
of $R_\k$ and $N_\k:=\bigcap_{j=1}^\nu\,I_j$. Then
$\nu=\nu(\P)\in\N$ and $N_\k$ is the maximal nilpotent ideal of
$R_\k$; see [\cite{Sch}, Thm.~2]. In particular,
$\bar{R}_\k:=R_\k/N_\k$ is a semiprime Noetherian ring. By Goldie's
theory, the set $\bar{S}_\k$ of all regular elements of $\bar{R}_\k$
is an Ore set in $\bar{R}_\k$ and the quotient ring
$\mathcal{Q}(\bar{R}_\k)=\bar{S}_\k^{-1}\bar{R}_\k$ is semisimple
and Artinian.

\smallskip

Write $Z(\bar{R}_\k)$ for the centre of $\bar{R}_\k$ and
$\mathcal{C}(Z(\bar{R}_\k))$ for the set of all elements of
$Z(\bar{R}_\k)$ which are regular in $\bar{R}_\k$. Since
$\bar{R}_\k$ is a finite module over the image of the $p$-centre of
$U(\g_\k)$ in $\bar{R}_\k$, it is algebraic over $Z(\bar{R}_\k)$.
Applying [\cite{AH}, Thm.~2] now yields that
$\mathcal{Q}(\bar{R}_\k)$ is obtained from $\bar{R}_\k$ by inverting
the elements from $\mathcal{C}(Z(\bar{R}_\k))$ (the latter is
obviously an Ore set in $\bar{R}_\k$).

\begin{prop}\label{quot}
There exists a unital subalgebra $\mathfrak{C}_\k$ of
$\mathcal{Q}(\bar{R}_\k)$ isomorphic to $\Mat_{l'}(\k)$ and such
that
$\mathcal{Q}(\bar{R}_\k)\,\cong\,\mathfrak{C}_\k\otimes\mathfrak{D}_\k$
where $\mathfrak{D}_\k$ is the centraliser of $\mathfrak{C}_\k$ in
$\mathcal{Q}(\bar{R}_\k)$.
\end{prop}
\begin{proof}
The ring theoretic notation used below will follow that of
[\cite{McR}]. Given a two-sided ideal $I$ of the ring $R_\k$ we
write $\mathcal{C}'(I)$ for the set of all elements $r\in R_\k$ for
which the coset $r+I$ is left regular in the ring $R_\k/I$ (the
latter means that $r\cdot x\in I$ for $x\in R_\k$ implies $x\in I$).
We denote by $\mathcal{C}(I)$ the set of all elements $r\in R_\k$
such that the coset $r+I$ is left and right regular in $R_\k/I$. As
we know, for each $y\in Y$ the element $y\otimes 1$ is regular in
$R_\k$. In particular, $y\otimes 1\in\mathcal{C}'(0)$. To ease
notation we now let $\bar{x}$ denote the image of $x\in R_A$ in
$\bar{R}_\k=R_\k/N_\k$. As the ring $R_\k$ is right Noetherian, it
follows from [\cite{Gol}, 2.3, 2.5] that $\mathcal{C}'(0)\subseteq
\mathcal{C}(N_\k)$ (see also [\cite{McR}, Prop.~4.1.3(iii)]). This
shows that for every $y\in Y$ the element $\bar{y}$ is regular in
$\bar{R}_\k$.

\smallskip

The subset $\bar{X}\cup\bar{Y}$ of $\bar{R}_\k$ contains elements
satisfying the relations (\ref{00}), (\ref{f}), (\ref{g}),
(\ref{h}), (\ref{i}), (\ref{j}), (\ref{m}), (\ref{n}), (\ref{p}),
(\ref{r'}), (\ref{s}).
 Since all elements of $\bar{Y}$
involved in these relations remain regular in $\bar{R}_\k$ and each
step of the procedure described in \ref{3.1} is reversible, we can
find elements $\bar{e}_{ij}$ and $\bar{c}_{ij}^k$ in
$\mathcal{Q}(\bar{R}_\k)$, where $1\le i,j\le l'$ and $1\le k\le n$,
satisfying the relations (\ref{a}), (\ref{b}), (\ref{d}), (\ref{e}).
We denote by $\mathfrak{C}_\k$ the $\k$-span of the
$\bar{e}_{ij}$'s. Thanks to (\ref{a}) and (\ref{b}), it is a
homomorphic image of $\Mat_{l'}(\k)$ and a unital subalgebra of
$\mathcal{Q}(\bar{R}_\k)$. Therefore,
$\mathfrak{C}_\k\cong\Mat_{l'}(\k)$ as $\k$-algebras.

\smallskip

In  view of (\ref{e}) all elements $\bar{c}_{ij}^k$ commute with
$\mathfrak{C}_\k$, whilst (\ref{d}) implies that the $\bar{g}_k$'s
lie in $\mathfrak{C}_\k\cdot \mathfrak{D}_\k$ where
$\mathfrak{D}_\k$ is the centraliser of $\mathfrak{C}_\k$ in
$\mathcal{Q}(\bar{R}_\k)$. As the inverses of the elements from
$\mathcal{C}(Z(\bar{R}_\k))$ lie in $\mathfrak{D}_\k$ as well and
$\mathcal{Q}(\bar{R}_\k)\,=\,\bar{S}_\k^{-1}\bar{R}_\k\,=\,
\big(\mathcal{C}(Z(\bar{R}_\k))\big)^{-1}\bar{R}_\k$ by our earlier
remarks, we deduce that
$\mathcal{Q}(\bar{R}_\k)\,=\,\mathfrak{C}_\k\cdot\mathfrak{D}_\k$.
As a consequence, there exists a surjective algebra homomorphism
$\psi\colon
\mathfrak{C}_\k\otimes\mathfrak{D}_\k\,\twoheadrightarrow\,
\mathcal{Q}(\bar{R}_\k)$. Since $\mathfrak{C}_\k$ is a matrix
algebra, it is straightforward to see that $\psi$ is injective. This
completes the proof.
\end{proof}
\subsection{}\label{4.2} Let $Z(\bar{R}_\k)$ be the centre of
$\bar{R}_\k$ and denote by $Z_p(\bar{R}_\k)$ the image of the
$p$-centre $Z_p(\g_\k)$ in $\bar{R}_\k$. Recall from (\ref{f-gen})
that the commutative $A$-algebra ${\rm gr}(R_A)$ is generated by $D$
homogeneous elements as a module over its graded polynomial
subalgebra $A[y_1,\ldots,y_{2d}]$, where $d=d(e)$.
\begin{lemma}\label{Z0}
There exists a $\k$-subalgebra $\bar{Z}_0$ of $Z_p(\bar{R}_\k)$
generated by $2d$ elements and such that $\bar{R}_\k$ is generated
as a $\bar{Z}_0$-module by $Dp^{2d}$ elements.
\end{lemma}
\begin{proof}
We follow the proof of [\cite{P07'}, Lemma~3.2] very closely. Write
$(R_A)_j$ (resp. $(R_\k)_j$) for the image in $R_A$ (resp. $R_\k$)
of the $j$th component of the canonical filtration of $U(\g_A)$
(resp. $U(\g_\k)$).

\smallskip

Suppose that $y_i$ has degree $a_i$, where $1\le i\le 2d$, and $v_k$
has degree $l_k$, where $1\le k\le D$, and let
$\Phi_A\colon\,S(\g_A)\twoheadrightarrow \gr(R_A)$ denote the
canonical homomorphism. For $1\le i\le 2d$ (resp. $1\le k\le D$)
choose $u_i\in U(\g_A)$ (resp. $w_k\in U_{l_k}(\g_A)$) such that
$\Phi_A(\gr_{a_i}\,u_i)=y_i$ (resp. $\Phi_A(\gr_{l_k}\,w_k)=v_k$).
Let $\bar{u}_i$ (resp. $\bar{w}_k$) denote the image of $u_i$ (resp.
$w_k$) in $R_\k=(U(\g_A)/\I_A)\otimes_A\k_\P$. For every $n\in\Z_+$
the set
$$\{w_ku_1^{i_1}\cdots u_{2d}^{i_{2d}}\,|\,\,
l_k+\textstyle{\sum}_{j=1}^{2d}\, i_ja_j\le n;\ 1\le k\le D\}$$
spans the $A$-module $(R_A)_n$. In view of our earlier remarks this
implies that the set
$$\{\bar{w}_k\bar{u}_1^{i_1}\cdots \bar{u}_{2d}^{i_{2d}}\,|\,\,
l_k+\textstyle{\sum}_{j=1}^{2d}\, i_ja_j\le n;\ 1\le k\le D\}$$
spans the $\k$-space $(R_\k)_n$. Since
$\gr_{pa_i}(\bar{u}_i^p)=(\gr_{a_i}\bar{u}_i)^p$ is a $p$th power in
$S(\g_\k)$, for every $i\le 2d$ there exists a $z_i\in
Z_p(\g_\k)\cap U_{a_i}(\g_\k)$ such that $\bar{u}_i^p-z_i\in
U_{pa_i-1}(\g_\k)$. We let $Z_0$ be the $\k$-subalgebra of
$Z_p(\g_\k)$ generated by $z_1,\ldots, z_{2d}$ and denote by
$\bar{Z}_0$ the image of $Z_0$ in $\bar{R}_\k=R_\k/N_\k$.

\smallskip

Let $R_\k'$ the $Z_0$-submodule of $R_\k$ generated by all
$\bar{w}_k\bar{u}_1^{i_1}\cdots \bar{u}_{2d}^{i_{2d}}$ with $0\le
i_j\le p-1$ and $1\le k\le D.$ Using the preceding remarks and
induction on $n$ we now obtain that $(R_\k)_n\subset R_\k'$ for all
$n\in\Z_+$. But then $R_\k=R_\k'$, implying that the set
$$\Lambda\,:=\,\{\bar{w}_k\bar{u}_1^{i_1}\cdots
\bar{u}_{2d}^{i_{2d}}\,|\,\, 0\le i_j\le p-1;\, 1\le k\le D\}$$
generates $R_\k$ as an $Z_0$-module. Obviously, $|\Lambda|\le
Dp^{2d}$. As $\bar{R}_\k$ is a homomorphic image of $R_\k$ and the
action of $\bar{Z}_0$ on $\bar{R}_\k$ is induced by that of
$Z_0\subset Z_p(\g_\k)$, the result follows.
\end{proof}
\begin{corollary}\label{max-dim}
Every irreducible $\bar{R}_\k$-module has dimension $\le
\sqrt{D}\cdot p^d$.
\end{corollary}
\begin{proof}
This is an immediate consequence of Lemma~\ref{Z0}, because the
central elements of $\bar{R}_\k$ act on any irreducible
$\bar{R}_\k$-module $V$ as scalar operators and the image of
$\Lambda$ in $\End V$ spans $\End V$.
\end{proof}
\begin{prop}\label{Krull}
The centre $Z(\bar{R}_\k)$ is an affine algebra over $\k$ and
$$\dim Z(\bar{R}_\k)=\dim Z_p(\bar{R}_\k)=2d.$$
\end{prop}
\begin{proof}
By Lemma~\ref{Z0}, $\bar{R}_\k$ is a finitely generated
$\bar{Z}_0$-module. Since $\bar{Z}_0$ is an affine $\k$-algebra,
$\bar{R}_\k$ is a Noetherian $\bar{Z}_0$-module. But then
$Z(\bar{R}_\k)$ and $Z_p(\bar{R}_\k)$ are finitely generated
$\bar{Z}_0$-modules. From this it is follows that the $\k$-algebra
$Z(\bar{R}_\k)$ is affine (of course, the same is true for
$Z_p(\bar{R}_\k)$, as it is a homomorphic image of $Z_p(\g_\k)$).
Both $Z(\bar{R}_\k)$ and $Z_p(\bar{R}_\k)$ being integral over
$\bar{Z}_0$, the inclusions $\bar{Z}_0\hookrightarrow
Z_p(\bar{R}_\k)$ and $\bar{Z}_0\hookrightarrow Z(\bar{R}_\k)$ give
rise to finite morphisms ${\rm Specm}\,\bar{Z}_0\twoheadrightarrow
{\rm Specm}\,Z_p(\bar{R}_\k)$ and ${\rm
Specm}\,\bar{Z}_0\twoheadrightarrow {\rm Specm}\,Z(\bar{R}_\k)$.
Since $\bar{Z}_0$ is a homomorphic image of the polynomial algebra
$\k[X_1,\ldots, X_{2d}]$, we now obtain
\begin{eqnarray}\label{leq}
\dim Z(\bar{R}_\k)=\dim Z_p(\bar{R}_\k)=\dim \bar{Z}_0\le 2d.
\end{eqnarray}

Recall from \ref{4.1} that the simple algebraic group $G_\k$ acts
rationally on $\bar{R}_\k$. Moreover, the canonical homomorphism
$c\colon\,U(\g_\k)\twoheadrightarrow R_\k=U(\g_\k)/\I_\k$ is
$G_\k$-equivariant. Since the inverse image under $c$ of the unique
maximal nilpotent ideal $N_\k$ of $R_\k$ is $G_\k$-stable, both
$Z_p(\bar{R}_\k)\cong Z_p(\g_\k)/\big(Z_p(\g_\k)\cap
c^{-1}({N}_\k)\big)$ and $Z(\bar{R}_\k)$ are stable under the action
of $G_\k$ on $\bar{R}_\k$. Since $Z_p(\bar{R}_\k)$ is a homomorphic
image of $Z_p(\g_\k)$, the maximal spectrum $\mathcal{V}_\P(M):={\rm
Specm}\,Z_p(\bar{R}_\k)$ identifies with a Zariski closed subset of
$\g_\k^*$ (see \ref{2.4} for more detail). By our discussion in
\ref{2.6}, the affine $G_\k$-variety $\mathcal{V}_\P(M)$ contains a
linear function $\Psi\in\chi+\m_\k^\perp$. Indeed, it is immediate
from the definition of $R_A$ that
$\widetilde{M}_{\k}=\widetilde{M}_A\otimes_A\k$ is an $R_\k$-module.
Therefore, $\bar{R}_\k=R_\k/N_\k$ acts on its simple quotient
$\widetilde{M}_{\k,\Psi}$. Since the $x^p-x^{[p]}$ acts on
$\widetilde{M}_{\k,\Psi}$ as $\Psi(x)^p\,{\rm Id}$ for every
$x\in\g_\k$ and $Z_p(\g_\k)\cap c^{-1}({N}_\k)$ annihilates
$\widetilde{M}_{\k,\Psi}$, we see that $\Psi$ induces an algebra
homomorphism $Z_p(\bar{R}_\k)\to\k$. In other words,
$\Psi\in\mathcal{V}_\P(M)$.

\smallskip

Given $j\in\Z^+$ we define $\Xi_j:=\{\eta\in\g_\k^*\,|\,\,\dim
\z_\xi\le 2j\}$, a Zariski closed, conical subset of $\g_\k^*$.
There is a cocharacter $\lambda\colon\,\k^\times\to G_\k$ such that
$(\Ad\lambda(t))(x)=t^ix$ for all $x\in \g_k(i)$ and
$t\in\k^\times$. Let $\rho_e\colon\,\k^\times \to\rm{GL}(\g_\k^*)$
denote the composition of ${\rm Ad}^*\,\lambda$ with the scalar
cocharacter $\xi\mapsto t^{-2}\xi$, where $\xi\in\g_\k^*$ and
$t\in\k^\times$. Obviously, $\rho_e$ induces a contracting
$\k^\times$-action on $\chi+\m_\k^\perp$ with centre at $\chi$.
Since for any $j$ the Zariski closed set $(\chi+\m_\k^\perp)\cap
\Xi_j$ is $\rho_e(\k^\times)$-stable and $\dim \z_\chi=2d$, we see
that $(\chi+\m_\k^\perp)\cap \Xi_j=\emptyset$ for all $j<2d$. This
implies that $\dim \z_\Psi\ge 2d$.

\smallskip

Since $({\rm Ad}^*\, G_\k)\,\Psi\subset \mathcal{V}_\P(M)$, we now
deduce that $\dim \mathcal{V}_\P(M)\ge 2d$. In conjunction with
(\ref{leq}) this gives $\dim Z(\bar{R}_\k)=\dim Z_p(\bar{R}_\k)=2d$,
as stated.
\end{proof}
\begin{rem}\label{Psi}
It follows from the proof of Proposition~\ref{Krull} that $\dim
\z_\Psi=2d$ and the orbit $({\rm Ad}^*\,G_\k)\,\Psi$ is open in the
variety $\mathcal{V}_\P(M)$. Moreover, arguing as in [\cite{P10},
3.6] it is easy to observe that $\chi$ and $\Psi$ belong to the same
sheet of $\g_\k^*$.
\end{rem}
\subsection{}\label{4.3} In this subsection we assume that
$\mathcal{D}_M$ is a Weyl skew-field and we adopt the notation and
conventions of \ref{4.1}. By Proposition~\ref{quot}, there is a
unital subalgebra $\mathfrak{C}_\k\cong\Mat_{l'}(\k)$ of
$\mathcal{Q}(\bar{R}_\k)$  such that
$\mathcal{Q}(\bar{R}_\k)\,\cong\,\mathfrak{C}_\k\otimes\mathfrak{D}_\k$
where $\mathfrak{D}_\k$ is the centraliser of $\mathfrak{C}_\k$ in
$\mathcal{Q}(\bar{R}_\k)$.
\begin{prop}\label{weyl-case}
Suppose $\mathcal{D}_M\cong \mathcal{Q}(\mathbf{A}_{d}(\CC))$ and
the admissible ring $A$ satisfies all the requirements of \ref{4.1}.
Then the $\k$-algebra $\mathfrak{D}_\k$ is isomorphic to the ring of
fractions $\mathcal{Q}(\mathbf{A}_d(\k))$ and
$\mathcal{Q}(\bar{R}_\k)\,\cong\,
\Mat_{l'}\big(\mathcal{Q}(\mathbf{A}_d(\k))\big)$.
\end{prop}
\begin{proof}
First recall from \ref{4.1} that given $x\in R_A$ we write $\bar{x}$
for the image of $x\otimes 1$ in
$\bar{R}_\k=(R_A\otimes_A\k_\P)/N_\k$. Repeating the argument used
at the beginning of the proof of Proposition~\ref{quot} we observe
that for every $y\in Y$ the element $\bar{y}$ is regular in
$\bar{R}_\k$.

\smallskip

The subset $\bar{X}\cup\bar{Y}$ of $\bar{R}_\k$ contains elements
satisfying the relations (\ref{uv}), (\ref{10}), (\ref{1'''}),
(\ref{2'''}), (\ref{13}), (\ref{06}), (\ref{18}), (\ref{18'}),
(\ref{18''}), (\ref{15}), (\ref{16}), (\ref{17}), (\ref{D1}),
(\ref{D2}), (\ref{D''}). Since all elements of $\bar{Y}$ involved in
these relations are regular and each step of the procedure described
in \ref{3.2} is reversible, we can find elements $w_1,\ldots,
w_{2d}$ in $\mathcal{Q}(\bar{R}_\k)$ satisfying the relations
(\ref{1}) and (\ref{2}). We denote by $\mathcal{D}'_\k$ the
$\k$-subalgebra of $\mathcal{Q}(\bar{R}_\k)$ generated by the
$w_i$'s. Clearly, $\mathcal{D}'_\k$ is a homomorphic image of the
Weyl algebra $\mathbf{A}_{d}(\k)$.

\smallskip

By (\ref{d'}), we have the inclusion $\mathcal{D}'_\k\subset
\mathfrak{D}_\k$.  Since the images of the $Q_{i;\,p(k)}^s$'s with
$i=1,2$ are regular in $\bar{R}_\k$ and
$\mathcal{Q}(\bar{R}_\k)\,=\,(\mathcal{C}(Z(\bar{R}_\k))^{-1}\bar{R}_\k$
by our earlier remarks, we can combine (\ref{qq}), (\ref{3}),
(\ref{d}) and (\ref{e}) with the equality
 $\mathcal{Q}(\bar{R}_\k)\,=\,\mathfrak{C}_\k\cdot\mathfrak{D}_\k$  to
 obtain
\begin{equation}\label{eqq}
\mathfrak{D}_\k\,=\,(\mathcal{C}(Z(\bar{R}_\k))^{-1}\mathcal{D}_\k'.
\end{equation}
Since it follows from Proposition~\ref{quot} that $\mathfrak{D}_\k$
is a semiprime ring, (\ref{eqq}) yields that $\mathcal{D}'_\k$ has
no nonzero nilpotent ideals, i.e. the ring $\mathcal{D}_\k'$ is
semiprime, too.

\smallskip

Let $\mathcal{C}(Z(\mathcal{D}'_\k))$ denote the set of all regular
elements of $\mathcal{D}_\k'$ contained in the centre of
$\mathcal{D}'_\k$. It is immediate from (\ref{eqq}) that
$\mathcal{C}(Z(\mathcal{D}'_\k))\subseteq \,
\mathcal{C}(Z(\mathfrak{D}_\k))$. So
$\mathcal{C}(Z(\mathcal{D}'_\k))$ is a multiplicative subset of
regular elements of $\mathcal{Q}(\bar{R}_\k)$ satisfying the left
and right Ore condition.

 Being a homomorphic image of $\mathbf{A}_d(\k)$ the $\k$-algebra
 $\mathcal{D}_\k'$ is finitely generated as a module over its centre.
 As $\mathcal{D}'_\k$ is a semiprime ring, applying
[\cite{AH}, Thm.~2] yields that $\mathcal{Q}(\mathcal{D}_\k')$ is
obtained from $\mathcal{D}'_\k$ by inverting the elements from
$\mathcal{C}(Z(\mathcal{D}'_\k))$. Combining this with (\ref{qq})
and (\ref{3}) we now deduce that
$\bar{c}_{ij}^k\in(\mathcal{C}(Z(\mathcal{D}'_\k)))^{-1}\mathcal{D}_\k'$
for all $1\le i,j\le l'$ and $1\le k\le n$. But then (\ref{d})
forces $\bar{g}_k\in
(\mathcal{C}(Z(\mathcal{D}'_\k)))^{-1}\mathfrak{C}_\k\cdot\mathcal{D}_\k'$
for all $1\le k\le n$. This, in turn, yields that
$(\mathcal{C}(Z(\mathcal{D}'_\k)))^{-1}\mathcal{D}_\k'$ contains
$\bar{R}_\k$ and hence $Z(\bar{R}_\k)$. Then our remarks earlier in
the proof show that
\begin{equation}\label{dk}\mathfrak{D}_\k\,=\,(\mathcal{C}(Z(\mathcal{D}'_\k)))^{-1}\mathcal{D}_\k'\,=\,
\mathcal{Q}(\mathcal{D}'_\k).\end{equation}

Let $Z_d(\k)$ denote the centre of the Weyl algebra
$\mathbf{A}_d(\k)$. It is well known and easily seen that
$\mathbf{A}_d(\k)$ is a free $Z_d(\k)$-module of rank $p^{2d}$ and
every two-sided ideal of $\mathbf{A}_d(\k)$ is centrally generated.
Furthermore, $Z_d(\k)$ is a polynomial algebra in $2d$ variables
over $\k$. Since $\mathcal{D}'_\k$ is a homomorphic image of
$\mathbf{A}_d(\k)$, its centre, $\bar{Z}_d$, is a homomorphic image
of $Z_d(\k)$. We let
$\beta\colon\,Z_d(\k)\twoheadrightarrow\,\bar{Z}_d$ denote the
corresponding homomorphism of $\k$-algebras.

\smallskip

Recall from \ref{4.1} that $N_\k=\bigcap_{i=1}^\nu I_i$ where
$I_1,\ldots, I_\nu$ are the minimal primes of $R_\k$. By the theory
of semiprime Noetherian PI-algebras finite over their centres, all
quotients $R_\k/I_j$ are prime and
$\mathcal{Q}(\bar{R}_\k)\,\cong\,\mathcal{Q}(R_\k/I_1)\oplus\cdots\oplus
\mathcal{Q}(R_\k/I_\nu)$ as $\k$-algebras. Moreover, each direct
summand $\mathcal{Q}(R_\k/I_j)$ is a simple algebra finite
dimensional over its centre $\mathcal{Q}(Z(R_\k/I_j))$; see
[\cite{R}], [\cite{AH}]. In particular, this shows that
$\mathcal{Q}(Z(\bar{R}_\k))=(\mathcal{C}(Z(\bar{R}_\k))^{-1}Z(\bar{R}_\k)$
injects into $\prod_{j=1}^\nu\,\mathcal{Q}(\bar{R}_\k/I_j)$, a
direct product of fields.

\smallskip

On the other hand, the algebra $Z(\bar{R}_\k)$ being reduced,
$\mathcal{Q}(Z(\bar{R}_\k))$ itself is a direct product of fields.
Furthermore, Proposition~\ref{Krull} implies that at least one of
the fields involved as direct factors of
$\mathcal{Q}(Z(\bar{R}_\k))$ has transcendence degree over $\k$
equal to $2d$. It follows that
\begin{equation}\label{tr}
{\rm tr.\,deg}_\k\,\mathcal{Q}(Z(R_\k/I_\ell)) = 2d\ \quad\mbox {for
some }\ \ell\le \nu.
\end{equation} Since
$\mathcal{Q}(\bar{R}_\k)\cong\mathfrak{C}_\k\otimes\mathcal{Q}(\mathcal{D}_\k')$,
it follows from our discussion earlier in the proof that
$\mathcal{Q}(Z(\bar{R}_\k))\cong\mathcal{Q}(\bar{Z}_d)$ as
$\k$-algebras. As the algebra $\bar{R}_\k$ is semiprime, its centre
$Z(\bar{R}_\k)$ is reduced and hence the ring of fractions
$\mathcal{Q}(\bar{Z}_d)$ is a direct product of fields. If
$\beta\colon\,Z_d(\k)\twoheadrightarrow \bar{Z}_d$ is not injective,
then $\dim \bar{Z}_d<2d$ and hence all fields involved as direct
factors of $\mathcal{Q}(\bar{Z}_d)$ have transcendence degree over
$\k$ less than $2d$. Since this contradicts (\ref{tr}), the map
$\beta$ must be injective. Then $\bar{Z}_d\,\cong\,Z_d(\k)$,
implying that
$\mathcal{Q}(\mathfrak{D}_\k)\cong\mathcal{Q}(\mathbf{A}_d(\k))$ and
$\mathcal{Q}(\bar{R}_\k)\cong\mathfrak{C}_\k\otimes
\mathcal{Q}(\mathbf{A}_d(\k))\cong
\Mat_{l'}\big(\mathcal{Q}(\mathbf{A}_d(\k))\big)$, as claimed.
\end{proof}
\begin{corollary}\label{prime} If
$\mathcal{D}_M\cong \mathcal{Q}(\mathbf{A}_{d}(\CC))$ as
$\CC$-algebras and the admissible ring $A$ satisfies all the
requirements of \ref{4.1}, then $\bar{R}_\k$ is a prime ring.
\end{corollary}
\begin{proof}
Since
$\mathcal{Q}(\bar{R}_\k)\,=\,\mathcal{C}(Z(\bar{R}_\k))^{-1}\bar{R}_\k$
and the ring $\mathcal{Q}(\mathbf{A}_{d}(\k))$ is prime, this is an
immediate consequence of Proposition~\ref{weyl-case}.
\end{proof}
\begin{conj}\label{conj}
We conjecture that under the above assumptions on $A$ the ring
$\bar{R}_\k$ is prime for any finite dimensional simple Lie algebra
$\g$ and any primitive ideal $I=I_M$.
\end{conj}
\subsection{}\label{4.4}
Write $\bar{I}_j$ for the image the minimal prime $I_j$ of $R_\k$ in
$\bar{R}_\k=R_\k/N_\k$. Since each quotient $\bar{R}_\k/\bar{I}_j$
is a prime ring, its central subalgebra
$Z_p(\bar{R}_\k)/\bar{I}_j\cap Z_p(\bar{R}_\k)$ is a domain. Since
the ${\rm PI}$-algebras module-finite over their Noetherian centres
enjoy the lying-over, going-up and incomparability properties for
prime ideals, every $\bar{I}_j\cap Z_p(\bar{R}_\k)$ is a minimal
prime of $Z_p(\bar{R}_\k)$ and every minimal prime of
$Z_p(\bar{R}_\k)$ is one of the $\bar{I}_j\cap Z_p(\bar{R}_\k)$'s;
see [\cite{McR}, Ch.~10] for more detail. It follows that there is
an $\ell\in\{1,\ldots,\nu\}$ such that $\dim
Z_p(\bar{R}_\k)\,=\,\dim Z_p(\bar{R}_\k)/\bar{I}_\ell\cap
Z_p(\bar{R}_\k).$ We now define
$\mathcal{R}:=\bar{R}_\k/\bar{I}_\ell$ and
$Z_p(\mathcal{R}):=Z_p(\bar{R}_\k)/\bar{I}_\ell\cap
Z_p(\bar{R}_\k)$. Then $\mathcal{R}$ is a prime Noetherian ring
which is finitely generated as a $Z_p(\mathcal{R})$-module.

\smallskip

Since $G_\k$ is a connected group, every minimal prime  $\bar{I}_j$
of $\bar{R}_\k$ is $G_\k$-stable. Therefore, $G_\k$ acts on
$\mathcal{R}$ as algebra automorphisms. Recall from \ref{4.2} the
Zariski closed set $\mathcal{V}_\P(M)\subset\g^*_\k$ which we have
identified with the maximal spectrum of $Z_p(\bar{R}_\k)$. As
explained in the proof of Proposition~\ref{Krull}, one of the
components of $\mathcal{V}_\P(M)$ contains a linear function
$\Psi\in\chi+\m_\k^\perp$ and $\dim ({\rm Ad}^*\,G)\,\Psi=2d$.

\smallskip

By construction, the zero locus of $\bar{I}_\ell\cap
Z_p(\bar{R}_\k)$ in $\mathcal{V}_\P(M)$ is an irreducible component
of maximal dimension in $\mathcal{V}_\P(M)$. Since $\dim
Z_p(\bar{R}_\k)=2d$ by Proposition~\ref{Krull} and all irreducible
components of $\mathcal{V}_\P(M)$ are $G_\k$-stable, we see that
$\Psi$, too, lies in an irreducible component of maximal dimension
of $\mathcal{V}_\P(M)$. But then the above discussion shows that we
can choose $\ell\in\{1,\ldots,\nu\}$ such that the zero locus of
$\bar{I}_\ell\cap Z_p(\bar{R}_\k)$ in $\mathcal{V}_\P(M)$ coincides
with the Zariski closure of $({\rm Ad}^*\,G)\,\Psi$ in $\g_\k^*$.
Therefore, no generality will be lost by assuming that $({\rm
Ad}^*\,G)\,\Psi$ is the unique open dense orbit of maximal spectrum
${\rm Specm}\,Z_p(\mathcal{R})\subset \g_\k^*$.

\smallskip

Since $\mathcal{R}$ is a Noetherian $Z_p(\mathcal{R})$-module, the
centre $Z(\mathcal{R})$ is finitely generated and integral over
$Z_p(\mathcal{R})$. Hence $Z_p(\mathcal{R})$ is an affine algebra
over $\k$ and the morphism $$\mu\colon\,{\rm
Specm}\,Z(\mathcal{R})\twoheadrightarrow {\rm
Specm}\,Z_p(\mathcal{R})$$ induced by inclusion
$Z_p(\mathcal{R})\hookrightarrow Z(\mathcal{R})$ is finite. In
particular, $\dim Z(\mathcal{R})=\dim Z_p(\mathcal{R})=2d$. As the
ring $\mathcal{R}$ is prime, the centre $Z(\mathcal{R})$ is a domain
and hence the affine variety $\mathcal{V}(\mathcal{R}):={\rm
Specm}\,Z(\mathcal{R})$ is irreducible. By our choice of $A$, the
rational action of the group $G_\k$ on $U(\g_\k)$ induces that on
$Z(\mathcal{R})$. Thus, $\mathcal{V}(\mathcal{R})$ is an irreducible
affine $G_\k$-variety.
\begin{prop}\label{orbit}
The following are true:

\begin{itemize}

\item[(i)\,]
The finite morphism
$\mu\colon\,\mathcal{V}(\mathcal{R})\twoheadrightarrow {\rm
Specm}\,Z_p(\mathcal{R})$ is $G_\k$-equivariant and the inverse
image of  $({\rm Ad}^*\,G)\,\Psi\subset {\rm
Specm}\,Z_p(\mathcal{R})$ under $\mu$ is a unique open dense
$G_\k$-orbit of $\mathcal{V}(\mathcal{R})$.

\smallskip

\item[(ii)\,]
The stabiliser $(G_\k)_c=\{g\in G_\k\,|\,\,g\cdot c=c\}$ of any
$c\in\mu^{-1}(\Psi)$ has the property that
$Z_{G_\k}(\Psi)^\circ\subseteq (G_\k)_c\subseteq Z_{G_\k}(\Psi)$.

\smallskip

\item[(iii)\,] The coadjoint stabiliser $Z_{G_\k}(\Psi)$ acts transitively on the
fibre $\mu^{-1}(\Psi)$.
\end{itemize}
\end{prop}
\begin{proof}
It is clear from our earlier remarks that $\mu$ is a finite morphism
equivariant under the action of $G_\k$. Let
$\mathcal{V}(\mathcal{R})_{\rm reg}$ denote the inverse image of
$({\rm Ad}^*\,G)\,\Psi$ under $\mu$. Since the map $\mu$ is
$G_\k$-equivariant, we have that $\mathcal{V}(\mathcal{R})_{\rm
reg}\,=\,\bigcup_{\,c\in\, \mu^{-1}(\Psi)}\,G_\k\cdot c.$ As the
morphism $\mu$ is finite, $\mu^{-1}(\Psi)$ is a finite set and $\dim
\mathcal{V}(\mathcal{R})\,=2d\,=\,({\rm Ad}^*\,G)\,\Psi$. From this
it is immediate that each orbit $G_\k\cdot c$ with $c\in
\mu^{-1}(\Psi)$ is Zariski open in $\mathcal{V}(\mathcal{R})$. As
the variety $\mathcal{V}(\mathcal{R})$ is irreducible, we see that
$(G_\k\cdot c)\cap (G_\k\cdot c')\neq \emptyset$ for any two
$c,c'\in \mu^{-1}(\Psi)$. This forces $G_\k\cdot c=G_\k\cdot c'$ for
all $c,c'\in \mu^{-1}(\Psi)$, implying that
$\mu^{-1}\big(\mathcal{V}(\mathcal{R})_{\rm reg}\big)\,=\,G_\k\cdot
c$ for any $c\in\mu^{-1}(\Psi)$. This proves statement~(i).

\smallskip

If $c\in\mu^{-1}(\Psi)$ and $g\in (G_\k)_c$, then
$$\Psi=\mu(c)=\mu(g\cdot c)=({\rm Ad}^*\,g)\,\mu(c)=({\rm
Ad}^*\,g)\,\Psi.$$ Therefore, $(G_\k)_c\subseteq Z_{G_\k}(\Psi)$. On
the other hand, the finite set $\mu^{-1}(c)$ is stable under the
action of $Z_{G_\k}(\Psi)$. As $G_\k$ acts regularly on the affine
algebraic variety $\mathcal{V}(\mathcal{R})$, it follows that the
stabiliser $(G_\k)_c$ of any $c\in\mu^{-1}(\Psi)$ is a Zariski
closed subgroup of finite index in $Z_{G_\k}(\Psi)$. So it must
contain the connected component of identity in $Z_{G_\k}(\Psi)$, and
statement~(ii) follows.

\smallskip

If $c, g(c)\in\mu^{-1}(\Psi)$ for some $g\in (G_\k)_c$, then
$\Psi=\mu(g(c))=g(\mu(c))=g(\Psi)$, forcing $g\in Z_{G_\k}(\Psi)$.
Thus, statement~(iii) is an immediate consequence of statement~(i).
\end{proof}
\begin{rem}
If $\mathcal{D}_M\cong\mathcal{Q}(\mathbf{A}_d(\CC))$, then
$\bar{R}_\k$ is a prime ring by Corollary~\ref{prime}. So in this
case we have that $\bar{R}_\k=\mathcal{R}$.
\end{rem}
\subsection{}\label{4.5}  Recall from [\cite{BS}], [\cite{R}],
[\cite{AH}] that any prime PI-ring $\mathcal A$ has a simple
Artinian ring of fractions $\mathcal{Q}(\mathcal{A})$ which
satisfies the same identities as $\mathcal{A}$ and is spanned by
$\mathcal{A}$ over its centre, $K$, which coincides with
$\mathcal{Q}(Z(\mathcal{A}))$. Moreover, $\dim_K\,
\mathcal{Q}(\mathcal{A})= d^2$, and after tensoring by a suitable
algebraic field extension $\tilde{K}$ of $K$, the ring
$\mathcal{Q}(\mathcal{A})$ becomes the matrix algebra
$\Mat_d(\tilde{K})$. Both $\mathcal{A}$ and
$\mathcal{Q}(\mathcal{A})$ satisfy all the polynomial identities of
$d\times d$ matrices over a commutative ring, but not those of
smaller matrices, and $d$ can be characterized as the least positive
integer such that $S_{2d}(X_1,\ldots,X_{2d})=0$ for all $X_1,\ldots,
X_{2d}\in\mathcal{A}$, where
$$S_{2d}(X_1,\ldots,X_{2d})\,:=\,\sum_{\sigma\in
\mathfrak{S}_{2d}}\,({\rm sgn}\,\sigma) X_{\sigma(1)}\cdots
X_{\sigma(2d)}.$$
\begin{defn}\label{deff}
The {\it PI-degree} of a prime PI-ring $\mathcal{A}$, denoted $
\mbox{PI-deg}(\mathcal{A})$, is defined as the least positive
integer $d$ such that $\mathcal{A}$ satisfies the {\it standard
identity} $S_{2d}\equiv 0$.
\end{defn}
\begin{defn}
We say that $\mathcal{A}$ is an {\it Azumaya algebra} over its
centre $Z(\mathcal{A})$ if $\mathcal{A}$ is a finitely generated
projective $Z(\mathcal{A})$-module and the natural map
$\mathcal{A}\otimes_{Z(\mathcal{A})}\mathcal{A}^{\rm op}\to\,{\rm
End}_{Z(\mathcal{A})}\,\mathcal{A}$ is an isomorphism.
\end{defn}
Now suppose that our PI-ring $\mathcal A$ is finitely generated over
its centre $Z(\mathcal{A})$ which, in turn, is an affine algebra
over $\k$. In this situation, it is known that
$\mbox{PI-deg}(\mathcal{A})=d(\mathcal{A})$, where $d(\mathcal{A})$
stands for the maximum $\k$-dimension of irreducible
$\mathcal{A}$-modules; see [\cite{BG}], for example. Let $V$ be an
irreducible $\mathcal{A}$-module, $P={\rm Ann}_{\mathcal{A}}\,V$ and
$\c=P\cap Z(\mathcal{A})$, a maximal ideal of $Z(\mathcal{A})$. It
follows from the Artin--Procesi theorem [\cite{McR}, Thm.~13.7.14]
that the equality $\dim V=\mbox{PI-deg}(\mathcal{A})$ holds if and
only if $\mathcal{A}_{\c}\,=\,\mathcal{A}\otimes_{Z(\mathcal{A})}
Z(\mathcal{A})_\c$ is an Azumaya algebra over the local ring
$Z(\mathcal{A})_\c$; see [\cite{BG}] for more detail.

\smallskip
The {\it Azumaya locus} of $\mathcal{A}$, denoted ${\rm
Az}(\mathcal{A})$, is defined as
$$
{\rm Az}(\mathcal{A}):=\{\c\in{\rm
Specm}\,Z(\mathcal{A})\,|\,\,\mathcal{A}_\c \, \mbox{ is an Azumaya
algebra}\,\}.
$$ The above discussion shows
that ${\rm Az}(\mathcal{A})$ consists of all $\c\in{\rm
Specm}\,Z(\mathcal{A})$ with
$\mathcal{A}_\c/\c\mathcal{A}_\c\,\cong\,\Mat_{d(\mathcal{A})}(\k)$,
whilst the Artin--Procesi theorem yields that ${\rm
Az}(\mathcal{A})$ is a nonempty Zariski open subset of ${\rm
Specm}\,Z(\mathcal{A})$; see [\cite{McR}, Thm.~13.7.14(iii)].
\subsection{}\label{4.6} In this subsection, we will prove
Theorems~{\rm A} and {\rm B}. First suppose that
$\mathcal{D}_M\cong\mathcal{Q}(\mathbf{A}_d(\CC))$. Then
Corollary~\ref{prime} says that $\bar{R}_\k=\mathcal{R}$ is a prime
ring. It follows from Proposition~\ref{weyl-case} that there exists
a finite algebraic extension $\widetilde{K}\cong\k(X_1,\ldots,
X_{2d})$ of the centre $K$ of $\mathcal{Q}(\mathcal{R})$ (identified
with $\mathcal{Q}(Z_{d}(\k))$, the centre of
$\mathcal{Q}(\mathbf{A}_d(\k))$) such that
$$\mathcal{Q}(\mathcal{R})\otimes_{K}\widetilde{K}\,\cong\,
\Mat_{l'}\big(\mathcal{Q}(\mathbf{A}_d(\k))\otimes_{K}\widetilde{K}\big)\,\cong
\,\Mat_{l'p^d}(\widetilde{K}).$$ As a result,
$\mbox{PI-deg}(\mathcal{R})=l'p^d$. On the other hand, since the
Azumaya locus of $\mathcal{R}$ is $G_\k$-stable and the dominant
morphism $\mu\colon\,{\rm
Specm}\,Z(\mathcal{R})=\mathcal{V}(\mathcal{R})\twoheadrightarrow\,{\rm
Specm}\,Z_p(\mathcal{R})$ from \ref{4.4} is $G_\k$-equivariant, it
must be that $\Psi\in \mu({\rm Az}(\mathcal{R}))$. But then
$\mu^{-1}(\Psi)\cap{\rm Az}(\mathcal{R})\neq \emptyset$. Applying
Proposition~\ref{orbit} now yields $\mu^{-1}(\Psi)\subset {\rm
Az}(\mathcal{R})$.

\smallskip

Let $\c(\Psi)$ denote the annihilator in $Z(\mathcal{R})$ of the the
irreducible $\mathcal{R}$-module $\widetilde{M}_{\k,\,\Psi}$
introduced in \ref{2.6}. Since $\c(\Psi)\in\mu^{-1}(\Psi)$, the
preceding remark shows that $\mathcal{R}_{\c(\Psi)}$ is an Azumaya
algebra. As $Z(\mathcal{R})_{\c(\Psi)}$ is a local ring, our
discussion in \ref{4.5} now yields that $\widetilde{M}_{\k,\,\Psi}$
is the only irreducible $\mathcal{R}_{\c(\Psi)}$-module (up to
isomorphism) and it has dimension equal to
$d(\mathcal{R})=\mbox{PI-deg}(\mathcal{R})$. Therefore,
$$l'p^{d}=\mbox{PI-deg}(\mathcal{R})=
\dim_\k\,\widetilde{M}_{\k,\,\Psi}=lp^d=(\dim_\CC\,M)p^d.$$ Since
$l'={\rm rk}\big(U(\g)/I_M\big)$, Theorem~{\rm B} follows.

\smallskip

It remains to prove Theorem~{\rm A}. Applying
Proposition~\ref{orbit} and arguing as before we obtain the
inclusion $\mu^{-1}(\Psi)\subset {\rm Az}(\mathcal{R})$ and hence
the equality $\mbox{PI-deg}(\mathcal{R})=lp^d$. On the other hand,
Proposition~\ref{quot} says that $\mathcal{Q}(\bar{R}_\k)\cong
\mathfrak{C}_\k\otimes\mathfrak{D}_\k$, where $\mathfrak{D}_\k$ is
the centraliser of $\mathfrak{C}_\k\cong\Mat_{l'}(\k)$ in
$\mathcal{Q}(\bar{R}_\k)$. Since $\mathcal{Q}(\bar{R}_\k)$ is a
semiprime Artinian ring, so is $\mathfrak{D}_\k$. Therefore,
$\mathfrak{D}_\k\cong \bigoplus_{j=1}^{\nu'}\,\mathfrak{D}_{\k,j}$
for some simple Artinian rings $\mathfrak{D}_{\k,j}$. But we know
that
$\mathcal{Q}(\bar{R}_\k)=\bigoplus_{j=1}^\nu\,\mathcal{Q}(\bar{R}_\k/\bar{I}_j)$
and each $\mathcal{Q}(\bar{R}_\k/\bar{I}_j)$ is a simple Artinian
ring; see our discussion in \ref{4.4}. Since
$\mathcal{Q}(\bar{R}_\k)\cong\bigoplus_{j=1}^{\nu'}\big(\mathfrak{C}_\k\otimes\mathfrak{D}_{\k,j}\big)$
and each $\mathfrak{C}_\k\otimes\mathfrak{D}_{\k,j}$ is a simple
Artinian ring, we now deduce that $\nu=\nu'$ and
$$\mathcal{Q}(\mathcal{R})\,=\,
\mathcal{Q}(\bar{R}_\k/\bar{I}_\ell)\,\cong\,\mathfrak{C}_\k\otimes\mathfrak{D}_{\k,\ell'}$$
for some $\ell'\le \nu$. As $\mathfrak{C}_\k\cong\Mat_{l'}(\k)$, our
discussion in \ref{4.5} then shows that $l'$ divides
$\mbox{PI-deg}(\mathcal{Q}(\mathcal{R}))=\mbox{PI-deg}(\mathcal{R})=lp^d$.
As $\Pi(A)$ contains almost all primes in $\N$, we can find $\P\in
{\rm Specm}\,A$ such that $l'$ is coprime to $p={\rm char}\,A/\P$.
Then we see that $l'={\rm rk}\big(U(\g)/I_M\big)$ must divide
$l=\dim_\CC\,M$, which completes the proof of Theorem~{\rm A}.
\begin{rem}\label{rr}
Let $\g=\mathfrak{sp}_{2n}$ with $n\ge 3$ and let $e\in\g$ be a
nilpotent element corresponding to partition $(2^n)$ of $2n$. We
have already mentioned in the Introduction that $I_{\rm
max}(\rho/2)$ is a completely prime primitive ideal of $U(\g)$ with
associated variety equal to $\overline{\O(e)}$. Hence $I_{\rm
max}(\rho/2)=I_M$ for some finite dimensional irreducible
$U(\g,e)$-module $M$. It is known (and not hard to see) that in the
present case the action of $\ad h$ on $\g$ gives rise to a short
$\Z$-grading
$$\g\,=\,\g(-2)\oplus\g(0)\oplus\g(2),\qquad\ \g(i)=\{x\in\g\,|\,\,[h,x]=ix\},$$
such that $\g(0)\cong\mathfrak{gl}(V)$ and $\g(2)\cong S^2(V)$ as
$\g(0)$-modules (here $V$ is an $n$-dimensional vector space over
$\CC$). Furthermore, $e\in S^2(V)$ corresponds to a {\it
nondegenerate} quadratic form on $V^*$ and the centraliser $\g_e$ is
a graded subspace of $\g(0)\oplus\g(2)$. Therefore,
$\g_e=\g_e(0)\oplus\g(2)$ and $\g_e(0)=\g_e\cap\g(0)$ is isomorphic
to $\mathfrak{so}(V)$ as Lie algebras. Since $\dim V\ge 3$ and $V$
is an irreducible $\mathfrak{so}(V)$-module, the Lie algebra
$\mathfrak{so}(V)$ is semisimple and the subspace of
$\mathfrak{so}(V)$-invariants in $S^2(V)$ is one-dimensional. This
implies that $\g_e=[\g_e,\g_e]\oplus\CC e$. But then the results
proved in [\cite{P07}, Sect.~2] show that the commutative quotient
$U(\g,e)^{\rm ab}$ of $U(\g,e)$ is generated by a Casimir element
$\Omega$ of $U(\g)$ (one should keep in mind here that $\Omega$ can
be regarded as one of the PBW generators of $U(\g,e)$, namely,
$\Omega=\Theta_e$). Since the Krull dimension of $U(\g,e)^{\rm ab}$
is positive by [\cite{P10}, Thm.~1.2] we now obtain that
$U(\g,e)^{\rm ab}\cong\,\CC[X]$ as algebras.

\smallskip

For $t\in\CC$ we denote by $L_t$ the $\g$-module induced from the
one-dimensional module $\CC\, v_t$ over the parabolic subalgebra
$\g(0)\oplus\g(2)$ of $\g$ such that $z\cdot v_t=tv_t$, where $z$ is
a fixed nonzero central element of the Levi subalgebra
$\g(0)\cong\mathfrak{gl}(V)$ of $\g$. By Conze's theorem, each ideal
$I_t:={\rm Ann}_{U(\g)}\,L_t$ of $U(\g)$ is completely prime (hence
prime). As the centre of $U(\g)$ acts on $L_t$ by scalar operators,
the Dixmier--M{\oe}glin equivalence yields that each $I_t$ is an
induced primitive ideal of $U(\g)$. Note that in the present case
$\m=\g(-2)$ and each $L_t$ is a free $U(\m)$-module of rank $1$.
Therefore, the subspace $M_t:=\{\psi\in
L_t^*\,|\,\,\psi(\m_\chi\cdot L_t)=0\}$ of the (full) dual space
$L_t^*$ is a one-dimensional $U(\g,e)$-module (recall that $\m_\chi$
is the subspace of $U(\g)$ spanned by all $x-\chi(x)$ with
$x\in\m$). We denote by $\widetilde{M}_t$ the submodule of the
$\g$-module $L_t^*$ generated by $M_t$.  Since ${\rm
Ann}_{U(\g)}\,V^*=({\rm Ann}_{U(\g)}\,V)^\top$ for any $\g$-module
$V$, for every $t\in\CC$ we have the inclusion  $I_{t}^\top\subseteq
{\rm Ann}_{U(\g)}\,\widetilde{M}_{t}$ (here $\scriptstyle{\top}$
stands for the principal anti-automorphism of $U(\g)$). The
Dixmier--Moerlin equivalence implies that each $I_t^\top$ is a
primitive ideal of $U(\g)$, whilst a routine verification shows that
${\rm gr}(I_t^\top)={\rm gr}(I_t)$ and hence
$\mathcal{VA}(I_t^\top)=\mathcal{VA}(I_t)$; see [\cite{P07}, 3.3]
for more detail.

\smallskip

It is immediate from Skryabin's equivalence that
$\widetilde{M}_{t}\cong Q_e\otimes_{U(\g,e)}M_{t}$ as $\g$-modules;
see [\cite{Sk}]. Then ${\rm Ann}_{U(\g)}\,\widetilde{M}_{t}=I_{M_t}$
and hence $\mathcal{VA}(I_t^\top)\supseteq \mathcal{VA}(I_{M_{t}})$.
Since each induced $\g$-module $L_t$ is holonomic and $e$ is a
Richardson element of $\g(0)\oplus\g(2)$, we also have that $\dim
\mathcal{VA}(I_t)=\dim \overline{\O(e)}=\dim
\mathcal{VA}(I_{M_{t}})$; see [\cite{P07}, Thm.~3.1(ii)]. In view of
the above remarks, this gives
$\mathcal{VA}(I_t^\top)=\mathcal{VA}(I_t)= \mathcal{VA}(I_{M_{t}})$.
As both $I_t^\top$ and $I_{M_t}$ are primitive ideals of $U(\g)$,
applying [\cite{BK}, Corollar~3.6] yields $I_t^\top=I_{M_{t}}$.
Since for $\g=\mathfrak{sp}_{2n}$ all homogeneous $(\Ad
G)$-invariants of $S(\g)$ have even degrees, one observes easily by
using the symmetrisation map $U(\g)\stackrel{\sim}{\to} S(\g)$ that
every set ${\mathcal X}_\lambda$ is stable under the principal
anti-automorphism $\scriptstyle{\top}$. The uniqueness of $I_{\rm
max}(\rho/2)\in{\mathcal X}_{\rho/2}$ then yields $I_{\rm
max}(\rho/2)=I_{\rm max}(\rho/2)^\top$.

\smallskip

On the other hand, it is well known that there exists a quadratic
polynomial $\phi\in\CC[X]$ such that the eigenvalue of $\Omega$ on
the parabolically induced $\g$-module $L_t$ equals $\phi(t)$ for any
$t\in\CC$. In conjunction with the above this shows that for every
$\varkappa\in\CC$ there is a $t(\varkappa)\in\CC$ such that $\Omega$
acts on $\widetilde{M}_{t(\varkappa)}$ as $\varkappa\,{\rm Id}$.
Since $U(\g,e)^{\rm ab}=\,\CC[\Omega]$ by our earlier remarks, we
now deduce that {\it all} primitive ideals $I_N$ associated with
one-dimensional $U(\g,e)$-modules $N$ belong to the set
$\{I_t^\top\,|\,\,t\in\CC\}$. Now $I_{\rm max}(\rho/2)=I_M$ for some
finite dimensional irreducible $U(\g,e)$-module $M$. If $\dim M=1$,
then $I_{\rm max}(\rho/2)^\top=I_{\rm max}(\rho/2)=I_{t_0}^\top$ for
some $t_0\in \CC$. But then $I_{\rm max}(\rho/2)=I_{t_0}$. Since the
primitive ideal $I_{\rm max}(\rho/2)=I_M$ is {\it not} induced by
[\cite{Mg2}, p.~45], we reach a contradiction thereby proving that
$\dim M>1$. As ${\rm rk}\big(U(\g)/I_{\rm max}(\rho/2)\big)=1$,
Theorem~{\rm B} implies that the Goldie field of the primitive
quotient $U(\g)/I_{\rm max}(\rho/2)$ is {\it not} isomorphic to a
Weyl skew-field. This example shows that Joseph's version of the
Gelfand--Kirillov conjecture fails for $\g=\mathfrak{sp}_{2n}$ with
$n\ge 3$.
\end{rem}
\begin{rem}\label{rr-r}
It seems that any attempt to generalise Theorem~{\rm B} would
require a rather detailed information on the structure of the Goldie
field $\mathcal{D}_M$. In the proof given above it was crucial for
us to know that $\mathcal{D}_M$ is generated as a skew-field over
$\CC$ by its $A$-subalgebra $\Delta$ with the following property:
for every $p\in\Pi(A)$ the $\k$-algebra
$\Delta_\k:=\Delta\otimes_A\k$ becomes the full matrix algebra
$\Mat_{p^d}(\widetilde{K})$ after a suitable algebraic field
extension $\widetilde{K}/\mathcal{Q}(Z(\Delta_\k))$ of the fraction
field of the centre of $\Delta_\k$. When $\mathcal{D}_M$ is a Weyl
skew-field, this condition is satisfied with the obvious choice of
$\Delta$, but the validity of Conjecture~\ref{conj} alone would not
guarantee that $\mathcal{D}_M$ contains an $A$-subalgebra $\Delta$
as above with the property that $\mbox{PI-deg}(\Delta_\k)=p^d$. The
best one could hope for would be an equality
$\mbox{PI-deg}(\Delta_\k)=sp^d$ with an explicit bound on $s$
independent of $p$. One might, for example, wonder whether $s$
always divides the order of the Weyl group $W=\langle
s_\gamma\,|\,\,\gamma\in\Phi\rangle$ or, more strongly, whether
$\mathcal{D}_M$ is always a crossed product of a Weyl skew-field
$\mathcal{F}$ and a finite group acting on $\mathcal{F}$ by algebra
automorphisms. Such unexplored possibilities make
Conjecture~\ref{conj} more flexible and hence harder to disprove
than the Gelfand--Kirillov conjecture for primitive quotients.
\end{rem}
\subsection{} Let $M$ and $M'$ be two generalised Gelfand--Graev
models of a primitive ideal $\I\in{\mathcal X}_\O$, so that
$\I=I_M=I_{M'}$. As we already mentioned in the Introduction, it was
conjectured by the author and proved by Losev in [\cite{Lo1}] that
$[M']=\,^{\!\gamma\!}[M]$ for some $\gamma\in\Gamma(e)$. We would
like to conclude this paper by showing that Conjecture~\ref{conj}
implies Losev's result.

\smallskip

Suppose $\bar{R}_\k$ is a prime ring and let $l=\dim V$, $l'=\dim
V'$. Let $\Gamma$ be a subset of $C(e)=G_e\cap G_f$ which maps
bijectively onto $\Gamma(e)$ under the canonical homomorphisms
$C(e)\to\Gamma(e)=C(e)/C(e)^\circ$. Let us assume for a
contradiction that $M'\not\cong\, ^{\!\gamma\!}M$ for any $\gamma\in
\Gamma$. Arguing as in \ref{2.6} we can find an admissible ring
$A\subset \CC$ and free $A$-submodules $M_A$ and $M'_A$ of $M$ and
$M'$, respectively, stable under $U(\g_A,e)$ and such that $M\cong
M_A\otimes_A \CC$ and $M'\cong M'_A\otimes_A\CC$. For every
$p\in\Pi(A)$ we then get $U(\g_\k,e)$-modules $M_\k=M_A\otimes_A\k$
and $M'_\k=M_A'\otimes_A\k$, where $\k=\overline{\Bbb F}_p$. As in
\ref{2.6} we localise further to reduce to the case where  $M_\k$
and $M'_\k$ are irreducible $U(\g_\k,e)$-modules for all
$p\in\Pi(A)$. Associated with $M_\k$ and $M'_\k$ are
$\bar{R}_\k$-modules $\widetilde{M}_{\k,\,\Psi}$ and
$\widetilde{M}'_{\k,\Psi'}$, where $\Psi,\Psi'\in\chi+\m_\k^\perp$;
see \ref{2.6} for more detail.

\smallskip

Recall from \ref{2.3} that $U(\g_A,e)$ is a free $A$-module with
basis consisting of the PBW monomials in $\Theta_1,\ldots,
\Theta_r$. Since $\Gamma$ is a finite set, we may assume (after
extending $A$ if necessary) that the $A$-form $U(\g_A,e)$ of
$U(\g,e)$ is stable under the action of the subgroup of $C(e)$
generated by $\Gamma$. Then each $^{\!\gamma\!}M_A$ with
$\gamma\in\Gamma$ can be regarded as a $U(\g_A,e)$-module. For
$\gamma\in\Gamma$, the equality ${\rm
Hom}_{U(\g,\,e)}(\,^{\!\gamma\!}M,M')=0$ comes down to the fact that
a certain homogeneous system of linear equations in $ll'$ unknowns
with coefficients in $A$ has no nonzero solutions. After inverting
in $A$ one of the  nonzero $ll'\times ll'$ minors of the matrix of
this homogeneous system we may assume that ${\rm
Hom}_{U(\g_\k,\,e)}(\,^{\!\gamma\!}M_\k,M_\k')=0$ for all
$p\in\Pi(A)$ and all $\gamma\in\Gamma$.

\smallskip

Recall from [\cite{P07'}] and [\cite{P10}] the subset $\pi(A)$ of
$\Pi(A)$; it consists of all primes $p\in\N$ such that
$A/\P\cong\mathbb{F}_p$ for some $\P\in{\rm Specm}\,A$. By
[\cite{P10}, Lemma~4.4], the set $\pi(A)$ is infinite.  The
preceding remark then shows that no generality will be lost by
assuming that $p\in\pi(A)$ and $^{\gamma\!}M_\k\not\cong M_\k'$ as
$U(\g_\k,e)$-modules for all $\gamma\in\Gamma$. Enlarging $A$
further if need be we may also assume that $\I={\rm
Ann}_{U(\g)}\,L(\lambda)$ and $\I_A\subseteq{\rm
Ann}_{U(\g_A)}\,L_A(\lambda)$ for some irreducible highest weight
module $L(\lambda)$ and that $A$ satisfies all the requirements of
[\cite{P10}, Sect.~4]. Since $\pi(A)$ is an infinite set, we may
also assume that the base change $A\to A/\P\hookrightarrow \k$
identifies $\Gamma\subset G(A)$ with a subset of $Z_{G_\k}(\chi)$
which maps {\it onto} the component group of $Z_{G_\k}(\chi)$ under
the canonical homomorphism $Z_{G_\k}(\chi)\to
Z_{G_\k}(\chi)/Z_{G_\k}(\chi)^\circ$.

\smallskip

Let $\P\in{\rm Specm}\,A$ be such that $A/\P\cong \mathbb{F}_p$. As
explained in [\cite{P10}, 4.5] the $R_\k$-module
$L_\P(\lambda)=L_A(\lambda)\otimes_A\k_\P$ has an irreducible
quotient, $L_\P^\eta(\lambda)$, which has $p$-character $\eta\in
({\rm Ad}^*\,G_\k)\,\chi$. As the ideal $N_\k$ is nilpotent,
$L_\P^\eta(\lambda)$ is an irreducible $\bar{R}_\k$-module. Since we
assume that the algebra $\bar{R}_\k$ is prime, the variety ${\rm
Specm}\,Z_p(\bar{R}_\k)\subset\g_\k^*$ is irreducible and $({\rm
Ad}^*\,G_\k)$-stable. By Proposition~\ref{Krull}, it has dimension
$2d$ which forces ${\rm Specm}\,Z_p(\bar{R}_\k)\,=\,\overline{({\rm
Ad}^*\,G_\k)\,\chi}$. But then both $\Psi$ and $\Psi'$ are $({\rm
Ad}^*\,G_\k)$-conjugate to $\chi$. As explained in Remark~\ref{rrr}
we can replace $\Psi$ and $\Psi'$ by their $({\rm
Ad}^*\,\mathcal{M}_\k)$-conjugates. In view of [\cite{P10},
Lemma~3.2] and standard properties of Slodowy slices, we therefore
may assume further that $\Psi=\Psi'=\chi$.

\smallskip

Denote by $\c$ and $\c'$ the annihilators in $Z(\bar{R}_\k)$ of
$\widetilde{M}_{\k,\,\Psi}$ and $\widetilde{M}'_{\k,\,\Psi'}$,
respectively. As $\mu(\c)=\mu(\c')=\chi$, Proposition~\ref{orbit}
shows that $\c'=\gamma_0(\c)$ for some $\gamma_0\in\Gamma$. On the
other hand, arguing as in \ref{4.6} it is straightforward to see
that $\c,\c'\in {\rm Az}(\bar{R}_\k)$. From this it follows that
$\widetilde{M}'_{\k,\,\Psi'}\,\cong\,^{\gamma_0\!}(\widetilde{M}_{\k,\,\Psi})$
as $\bar{R}_\k$-modules and hence as $U_\chi(\g_\k)$-modules. In
view of the Morita equivalence mentioned in \ref{2.5} this implies
that $$M_\k'\,\cong\, {\rm
Wh}_\chi(\widetilde{M}'_{\k,\,\Psi'})\,\cong\,{\rm Wh}_\chi\,
^{\gamma_0\!}(\widetilde{M}_{\k,\,\Psi})\,\cong\,{\rm
Wh}_\chi\,^{\gamma_0\!}\big(Q_e^\chi\otimes_{U_\chi(\g_\k,\,e)}\,M_\k\big)$$
as $U(\g_\k,e)$-modules. The adjoint action of $Z_{G_\k}(\chi)$ on
$\g_\k$ gives rise to a natural group homomorphism
$Z_{G_\k}(\chi)\to {\rm Aut}\,U_\chi(\g_\k)$. By [\cite{P02},
Thm.~2.3(i)], the left regular module $U_\chi(\g_\k)$ is isomorphic
to a direct sum of $p^d$ copies of $Q_e^\chi$. Since the same is
true with $\gamma_0(Q_e^\chi)\subset U_\chi(\g_\k)$ in place of
$Q_e^\chi$, we see that the projective $U_\chi(\g_\k)$-modules
$Q_\chi^e$ and $^{\gamma_0\!}(Q_e^\chi)\,\cong\,\gamma_0(Q_\chi^e)$
are isomorphic. Note that the right action of $U_\chi(\g_\k,e)$ on
$^{\gamma_0\!}(Q_e^\chi)$ is the $\gamma_0$-twist of that on
$Q_e^\chi$, where $\gamma_0\in\Gamma$ is now regarded as an
automorphism of $U(\g_\k,e)$. Comparing common
$(\Ad\gamma_0)(\m_\k)$-eigenvectors one observes that the
$U_\chi(\g_\k)$-modules $^{\gamma_0\!}(\widetilde{M}_{\k,\,\Psi})$
and
$^{\gamma_0\!}(Q_e^\chi)\otimes_{U_\chi(\g_\k,\,e)}{^{\gamma_0\!}}M_\k$
are isomorphic. As a consequence,
$$^{\gamma_0\!}(\widetilde{M}_{\k,\,\Psi})
\cong\,^{\gamma_0\!}(Q_e^\chi)\otimes_{U_\chi(\g_\k,\,e)}{^{\gamma_0\!}}M_\k
\,\cong\,Q_e^\chi\otimes_{U_\chi(\g_\k,\,e)}{^{\gamma_0\!}M_\k}
$$ as $U_\chi(\g_\k)$-modules.
In view of the above-mentioned Morita equivalence this entails that
$M_\k \cong  {^{\gamma_0\!} M_\k}$ as $U(\g_\k,e)$-modules.

\smallskip

We have reached a contradiction thereby showing that $M'\cong\,
^{\!\gamma\!}M$ for some $\gamma\in\Gamma$.

\end{document}